\documentclass[a4paper,11pt]{amsart}
\title[Generalized Donaldson-Thomas invariants]{Generalized Donaldson-Thomas invariants on the local projective plane}
\date{}
\author{Yukinobu Toda}

\usepackage{amscd}
\usepackage{amsmath}
\usepackage{amssymb}
\usepackage{amsthm}
\usepackage{float}
\usepackage[dvips]{graphicx}

\usepackage[all,ps,dvips]{xy}

\usepackage{array}
\usepackage{amscd}
\usepackage[all]{xy}

\DeclareFontFamily{U}{rsfs}{%
\skewchar\font127}
\DeclareFontShape{U}{rsfs}{m}{n}{%
<-6>rsfs5<6-8.5>rsfs7<8.5->rsfs10}{}
\DeclareSymbolFont{rsfs}{U}{rsfs}{m}{n}
\DeclareSymbolFontAlphabet
{\mathrsfs}{rsfs}
\DeclareRobustCommand*\rsfs{%
\@fontswitch\relax\mathrsfs}

\theoremstyle{plain}
\newtheorem{thm}{Theorem}[section]
\newtheorem{prop}[thm]{Proposition}
\newtheorem{lem}[thm]{Lemma}

\newtheorem{defi}[thm]{Definition}
\newtheorem{rmk}[thm]{Remark}
\newtheorem{cor}[thm]{Corollary}

\newtheorem{prop-defi}[thm]{Proposition-Definition}
\newtheorem{thm-defi}[thm]{Theorem-Definition}
\newtheorem{lem-defi}[thm]{Lemma-Definition}

\newdimen\argwidth
\def\db[#1\db]{
 \setbox0=\hbox{$#1$}\argwidth=\wd0
 \setbox0=\hbox{$\left[\box0\right]$}
  \advance\argwidth by -\wd0
 \left[\kern.3\argwidth\box0 \kern.3\argwidth\right]}

\newcommand{\bB}{\mathcal{B}}

\newcommand{\hH}{\mathcal{H}}

\newcommand{\mM}{\mathcal{M}}

\newcommand{\oO}{\mathcal{O}}

\newcommand{\uU}{\mathcal{U}}

\newcommand{\xX}{\mathcal{X}}
\newcommand{\yY}{\mathcal{Y}}

\newcommand{\Hom}{\mathop{\rm Hom}\nolimits}

\newcommand{\NS}{\mathop{\rm NS}\nolimits}

\newcommand{\Pic}{\mathop{\rm Pic}\nolimits}

\newcommand{\ch}{\mathop{\rm ch}\nolimits}

\newcommand{\Ext}{\mathop{\rm Ext}\nolimits}
\newcommand{\sgn}{\mathop{\rm sgn}\nolimits}

\newcommand{\rank}{\mathop{\rm rank}\nolimits}
\newcommand{\Coh}{\mathop{\rm Coh}\nolimits}

\newcommand{\cneq}{\mathrel{\raise.095ex\hbox{:}\mkern-4.2mu=}}
\newcommand{\eqcn}{\mathrel{=\mkern-4.5mu\raise.095ex\hbox{:}}}

\newcommand{\wP}{\widehat{\mathbb{P}}^2}

\newcommand{\DT}{\mathop{\rm DT}\nolimits}
\newcommand{\Eu}{\mathop{\rm Eu}\nolimits}

\newcommand{\Imm}{\mathop{\rm Im}\nolimits}

\newcommand{\cl}{\mathop{\rm cl}\nolimits}

\begin{document}
\maketitle

\begin{abstract}
We show that the generating series of 
generalized Donaldson-Thomas invariants 
on the local projective plane with any positive 
rank is 
described in terms of 
modular forms 
and theta type series for
indefinite lattices.  
In particular it absolutely converges to give a 
holomorphic function on the upper half plane. 
\end{abstract}

\section{Introduction}
\subsection{Motivation}
Let 
\begin{align*}
\pi \colon X \to \mathbb{P}^2
\end{align*}
be the total space of the canonical line bundle on $\mathbb{P}^2$. 
The space $X$ is a non-compact Calabi-Yau 3-fold, and 
the enumerative invariants (e.g. Gromov-Witten invariants, Donaldson-Thomas invariants) on $X$ have drawn attention 
in connection with string theory. 
Among such invariants, we focus 
on the \textit{generalized Donaldson-Thomas (DT) invariants}
introduced by Thomas~\cite{Thom}, Joyce-Song~\cite{JS}
and Kontsevich-Soibelman~\cite{K-S}. 
Given an element 
\begin{align*}
(r, l, \Delta) \in \mathbb{Z}^{\oplus 3}
\end{align*}
the generalized DT
invariant
\begin{align}\label{intro:DT}
\DT(r, l, \Delta) \in \mathbb{Q}
\end{align}
counts 
 semistable\footnote{The invariant (\ref{intro:DT}) is independent of 
a choice of a stability, i.e. slope stability 
or Gieseker stability. See Lemma~\ref{lem:compare}. } sheaves $E$ on $X$ 
supported on the zero section\footnote{Indeed such sheaves
are scheme theoretically supported on the zero section. See Lemma~\ref{rmk:standard}.} of $\pi$, 
satisfying 
\begin{align}\label{intro:cond}
\rank(\pi_{\ast}E)=r, \ 
c_1(\pi_{\ast}E)=l, \ 
\Delta(\pi_{\ast}E)=\Delta. 
\end{align}
 Here $\Delta(\pi_{\ast}E)$
is the discriminant
\begin{align*}
\Delta(\pi_{\ast}E)
=l^2 -2r \ch_2(\pi_{\ast}E). 
\end{align*}
We are interested in the generating series:
\begin{align}\label{intro:gen}
\DT(r, l) \cneq \sum_{\Delta \in \mathbb{Z}_{\ge 0}}
\DT(r, l, \Delta)(-q^{\frac{1}{2r}})^{\Delta}. 
\end{align}
If $r$ and $l$ are coprime, 
then the 
series (\ref{intro:gen}) is the generating 
series of Euler numbers of 
moduli spaces of stable sheaves on $\mathbb{P}^2$, 
which 
has been explicitly computed up to rank three in several 
literatures~\cite{Got}, \cite{Kly}, \cite{YosB}, \cite{Yo1}, 
\cite{GoTheta}, \cite{Man2}, 
\cite{BrMan}, \cite{Kool}, \cite{Wei}.
If $r$ and $l$ are not coprime, 
then the definition of the 
invariant (\ref{intro:DT}) involves the logarithm 
of the moduli stack in the Hall algebra, and its
explicit computation is more subtle. 
Nevertheless there exist works~\cite{Man1}, \cite{GS1}
in which the series (\ref{intro:gen}) is studied 
for non-coprime $(r, l)$
up to rank three. 
In any case, 
the resulting closed formula is quite complicated 
even in the rank three case\footnote{For instance,
a formula in the rank three case
occupies 1.5 pages in~\cite[Section~4.3]{Kool}.}, 
and it seems hopeless to obtain a neat closed formula 
for an arbitrary rank. 

On the other hand, 
by Vafa-Witten's S-duality conjecture~\cite{VW}, the series (\ref{intro:gen})
is expected to have a certain modular invariance property. 
The computation of the series (\ref{intro:gen}) in 
the rank two case~\cite{Kly}, \cite{BrMan}
indicates that (\ref{intro:gen}) is not 
a modular form in a strict sense, but may be so 
in a broad sense including \textit{mock modular forms}~\cite{Zweg}, 
\cite{Zagier}. 
In order to approach the S-duality
 conjecture, we may not have to worry 
about the complexity of the explicit closed formula:
it is enough to know that the series (\ref{intro:gen}) is a finite
linear combination of modular forms of the same weight
in a broad sense.  
The purpose of this paper is to show that the series (\ref{intro:gen}) 
for any $r\ge 1$ is always written in terms of 
modular forms and certain theta type series for indefinite lattices,  
which converge and hopefully have a modular invariance property
in a broad sense. 

\subsection{Main result}
We construct theta type series from data
\begin{align}\label{intro:xi}
\xi=(\Gamma, B, \overline{\nu}, c_1, c_2, \cdots, c_b, c_1', 
c_2', \cdots, c_b', \alpha_1, \cdots, \alpha_k). 
\end{align}
Here $(\Gamma, B(-, -))$ is a non-degenerate lattice 
with index $(a, b)$, 
and 
$\overline{\nu}, c_i, c_i', \alpha_i$
are elements of $\Gamma_{\mathbb{Q}}$, satisfying 
certain conditions described in Subsection~\ref{subsec:theta}.  
Given data (\ref{intro:xi}), 
we construct 
the series
\begin{align}\label{intro:Theta}
&\Theta_{\xi}(q) 
\cneq \\
\notag
&\sum_{\nu \in \overline{\nu} + \Gamma}
\prod_{i=1}^{b}\left( \sgn (B(c_i, \nu))-
\sgn (B(c_i', \nu)) \right)
\prod_{j=1}^{k} B(\alpha_j, \nu) \cdot q^{Q(\nu)}. 
\end{align}
Here $Q(\nu)=B(\nu, \nu)/2$ and $\sgn(x)$
is defined by 
\begin{align*}
\sgn(x)=\left\{ \begin{array}{cl}
x/\lvert x \rvert & \text{ if }x\neq 0 \\
0 & \text{ if } x=0. 
\end{array}\right. 
\end{align*}
The series (\ref{intro:Theta})
is a generalization of known theta type series, 
and our conditions in Subsection~\ref{subsec:theta} allow us to show 
the convergence  (cf.~Lemma~\ref{lem:converge}) of (\ref{intro:Theta})
after the substitution 
 $q=e^{2\pi i \tau}$, where $\tau \in \hH \subset \mathbb{C}$
and $\hH$ is the upper half plane. 
For example if $b=k=0$, 
then $Q$ is a positive definite quadratic form on $\Gamma$, and 
the series (\ref{intro:Theta}) is nothing but the 
classical theta series. 
In this case, we call 
data (\ref{intro:xi}) as \textit{classical data}. 
If $b=1$ and $k=0$, then the series (\ref{intro:Theta})
is a mock theta series studied by Zwegers~\cite{Zweg}. 
Some more detail on the series (\ref{intro:Theta}) will 
be discussed in Subsection~\ref{subsec:theta}. 
The following is the main result in this paper: 

\begin{thm}\emph{(Corollary~\ref{cor:main}.)} \label{intro:main}
For any $r\in \mathbb{Z}_{\ge 1}$ and $l\in \mathbb{Z}$, there is 
a finite number of data $\xi_1, \cdots, \xi_n$ as in (\ref{intro:xi}),
classical data $\xi'$, $a_1, \cdots, a_n \in \mathbb{Q}$
 and $N\in \mathbb{Z}_{\ge 1}$
such that the following holds: 
\begin{align*}
\DT(r, l)
=q^{\frac{r}{8}}\eta(q)^{-3r} \cdot \Theta_{\xi'}(q)^{-1}
\cdot \left(\sum_{i=1}^{n} a_i \Theta_{\xi_i}(q^{\frac{1}{N}})  \right).
\end{align*} 
Here $\eta(q)$ is the Dedekind eta function
\begin{align}\label{eta}
\eta(q)=q^{\frac{1}{24}}\prod_{m\ge 1}(1-q^m). 
\end{align}
\end{thm}

Since the series (\ref{intro:Theta}) 
converges, we obtain the following 
 corollary:
\begin{cor}
For any $r\in \mathbb{Z}_{\ge 1}$ and $l\in \mathbb{Z}$, 
the generating function
\begin{align*}
\sum_{\Delta\in \mathbb{Z}_{\ge 0}} \DT(r, l, \Delta)
e^{2\pi i \Delta \tau}
\end{align*}
converges absolutely on the upper half plane $\tau \in \hH \subset \mathbb{C}$.
\end{cor}
The strategy of the proof of Theorem~\ref{intro:main}
is described below. 
Although it 
follows from a
 traditional approach, the result of Theorem~\ref{intro:main}
is a new structure result for the series (\ref{intro:gen})
with an arbitrary positive rank\footnote{In~\cite[Theorem~3.7]{Kool}, 
by the torus localization, 
Kool described the Euler numbers of moduli spaces of 
stable sheaves on $\mathbb{P}^2$ with any positive 
rank in terms of Euler numbers of certain explicit 
varieties given by GIT quotients.
Since the computation of the latter numbers is not obvious, 
his result does not imply Theorem~\ref{intro:main} even 
if $r$ and $l$ are coprime. }.

\subsection{Strategy of the proof of Theorem~\ref{intro:main}}
\label{subsec:past}
So far there have been two kinds of approaches toward
the study of the invariants (\ref{intro:DT}): 
one is to use the localization with respect to 
the torus action~\cite{Kly}, 
\cite{Kool}, \cite{Wei}, \cite{GS1}, 
and the other one is to use the blow-up 
formula and the wall-crossing formula~\cite{Yo1}, \cite{Man2}, \cite{Man1}, 
\cite{BrMan}. 
We follow the latter strategy. 
In fact, the latter one has been used to compute 
Betti numbers (rather than Euler numbers)
of moduli spaces of stable sheaves on $\mathbb{P}^2$. 
Let 
\begin{align*}
f \colon \wP \to \mathbb{P}^2
\end{align*}
 be a blow-up at a point
and $C$ the exceptional divisor of $f$. 
The
blow-up formula~\cite{Yo1}, \cite{WZ}, \cite{GoTheta}
describes  
Betti numbers of the moduli spaces of stable sheaves on $\mathbb{P}^2$
in terms of those
on $\wP$ with respect to the $f^{\ast}H$-stability
and classical theta series.  
Here $H$ is the hyperplane class of $\mathbb{P}^2$. 
Note that 
$\wP$ admits a $\mathbb{P}^1$-fibration $\wP \to \mathbb{P}^1$, 
and we denote by $F$ a fiber class. 
Let us consider a one parameter family of 
$\mathbb{R}$-divisors on $\wP$: 
\begin{align*}
H_t=f^{\ast}H -tC, \ t \in [0, 1).
\end{align*} 
The $\mathbb{R}$-divisor $H_t$ is ample for $t\in (0, 1)$. 
It is well-known that, for $t$ sufficiently close to $1$, 
there is no $H_t$-semistable sheaf $E$ on $\wP$ with $\rank(E) \ge 2$
and $c_1(E) \cdot F=1$.
This
fact together with the 
wall-crossing from $H_0$ to $H_t$ with $t\to 1-0$
enable us to describe 
Betti numbers of moduli spaces of
 $H_0$-semistable sheaves on $\wP$
in terms of those with lower rank. 
Combined with the blow-up formula, we 
can compute the desired Betti numbers on $\mathbb{P}^2$
by the induction of the rank. 
The above argument was considered by Yoshioka~\cite{Yo1}
in the rank two case, by Manschot~\cite{Man2}, \cite{Man1}
in the rank three case. 
As pointed out in~\cite{Man1}, 
the argument in principal can be applied for an arbitrary rank. 
However there are some issues to apply 
the above arguments to study the series (\ref{intro:gen}): 
\begin{itemize}
\item If $r$ and $l$ are not coprime, then a relationship 
between Betti numbers of the moduli spaces\footnote{
In this case, the moduli space is an algebraic stack, 
and its Betti numbers are interpreted as a rational function 
given by the ratio of Poincar\'e polynomials.}
and the generalized DT invariants is not yet established. 
\item To obtain a result for the series (\ref{intro:gen}) from 
the result of Betti numbers, one has to take a specialization, 
whose computation is not obvious. 
\item The wall-crossing formula is quite complicated, and it is hard
to describe the result in a general rank. 
\end{itemize}
In order to avoid the first and the second issues, 
we directly work with 
the generalized DT invariants, rather than Betti numbers.  
Instead of using the results of~\cite{Yo1}, \cite{WZ}, \cite{GoTheta}, 
we use the result of~\cite{TodS} in which a blow-up formula 
for the series (\ref{intro:gen}) was obtained by interpreting 
a blow-up of a surface as a 3-fold flop. 
As for the third issue, we 
work with Joyce's wall-crossing coefficients~\cite{Joy4} 
with respect to the polarization change from $H_0$ to $H_t$ with $t\to 1-0$ in 
detail, and extract the theta type series (\ref{intro:Theta}).

\subsection{Acknowledgment}
This work is supported by World Premier 
International Research Center Initiative
(WPI initiative), MEXT, Japan. This work is also supported by Grant-in Aid
for Scientific Research grant (No. 26287002)
from the Ministry of Education, Culture,
Sports, Science and Technology, Japan.

\section{Preliminary}
This section is devoted to a preliminary to the proof of 
Theorem~\ref{intro:main}. 
Throughout this paper, all the varieties or stacks
are defined over $\mathbb{C}$. 
\subsection{Stability conditions on local surfaces}
Let $S$ be a smooth projective surface and 
\begin{align*}
\pi \colon 
X=\omega_{S} \to
S
\end{align*}
the total space of the canonical line bundle 
on $S$. 
Note that $X$ is a non-compact Calabi-Yau 3-fold, i.e.
$\omega_X \cong \oO_X$. 
Let
\begin{align*}
\Coh_c(X) \subset \Coh(X)
\end{align*}
be the abelian category of 
coherent sheaves on $X$
whose supports are compact. 
We recall two kinds of stability conditions on 
$\Coh_c(X)$ which depend on a choice of an ample 
$\mathbb{R}$-divisor $H$ on $X$:
slope stability
condition and Gieseker
stability condition. 

The slope stability condition uses 
the following slope function for $0\neq E \in \Coh_c(X)$: 
\begin{align*}
\mu_H(E) =\frac{c_1(\pi_{\ast}E) \cdot H}{\rank(\pi_{\ast}E)} \in 
\mathbb{R} \cup \{ \infty\}. 
\end{align*}
Here we set $\mu_H(E)=\infty$ if $\rank(\pi_{\ast}E)=0$. 
\begin{defi}
A pure two dimensional sheaf $E \in \Coh_c(X)$ is $H$-slope
(semi)stable if for any 
short exact sequence $0\to F \to E \to G \to 0$ in 
$\Coh_c(X)$ with $F, G \neq 0$, 
we have 
$\mu_H(F)<(\le) \mu_H(G)$. 
\end{defi}

The Gieseker stability condition uses 
the reduced Hilbert polynomial
for $E \in \Coh_c(X)$\footnote{By the Riemann-Roch theorem, we
can formally define (\ref{overchi}) for any $\mathbb{R}$-divisor $H$. }:
\begin{align}\label{overchi}
\overline{\chi}_H(E)=
\chi(E \otimes \oO_X(mH))/a_d
\end{align}
where $a_d$ is the leading coefficient of 
$\chi(E\otimes \oO_X(mH))$. 
We write $\overline{\chi}_H(F) \prec \overline{\chi}_H(E)$
if $\overline{\chi}_H(F)<\overline{\chi}_H(E)$ for $m\gg 0$. 
\begin{defi}
A pure two dimensional sheaf 
 $E \in \Coh_c(X)$ is Gieseker (semi)stable
if for any non-zero proper subsheaf $F \subset E$, 
we have $\overline{\chi}_H(F) \prec (\preceq) \overline{\chi}_H(E)$. 
\end{defi}
We
have the obvious implications: 
\begin{align}\notag
\mbox{slope stable} &\Rightarrow \mbox{Gieseker stable} \\
\label{obvious}
&\Rightarrow \mbox{Gieseker semistable} \Rightarrow
\mbox{slope semistable}. 
\end{align}
We regard $S$ as a closed subscheme of $X$ by the zero section of $\pi$. 
In some situation, an object in $\Coh_c(X)$ is supported on $S$: 
\begin{lem}\label{rmk:standard}
Suppose that $K_S \cdot H <0$. 
Then any $H$-slope semistable sheaf $E \in \Coh_c(X)$ 
is an $\oO_S$-module. In particular, any 
pure two dimensional sheaf on $X$ is supported on $S$. 
\end{lem} 
\begin{proof}
Applying $\otimes_{\oO_X}E$ to the inclusion 
$\oO_X(-S) \subset \oO_X$, 
we obtain the map 
\begin{align*}
E(-S) \to E. 
\end{align*}
The above map is zero since 
\begin{align*}
\mu_H(E(-S))=\mu_H(E)-K_S \cdot H >\mu_H(E)
\end{align*}
and $E$, $E(-S)$ are $H$-slope semistable. 
This implies that $E$ is an $\oO_S$-module. 
(Also see~\cite[Lemma~2.1]{GS1}.)
\end{proof}

\subsection{Hall algebras}
We recall the stack theoretic Hall algebras 
of $\Coh_c(X)$ introduced by Joyce~\cite{Joy4}. 
Let $\mathfrak{M}$ be the moduli 
stack of all the objects in $\Coh_c(X)$. 
The stack theoretic Hall algebra $H(X)$ is 
$\mathbb{Q}$-spanned by the isomorphism 
classes of the symbols (cf.~\cite{Joy4})
\begin{align}\label{symbol}
[\rho \colon \xX \to \mathfrak{M}]
\end{align}
where $\xX$ is an algebraic stack of finite type with 
affine geometric stabilizers and
$\rho$ is a 1-morphism. 
The relation is generated by 
\begin{align}\label{relation}
[\rho \colon \xX \to \mathfrak{M}]
\sim [\rho|_{\yY} \colon \yY \to \mathfrak{M}]
+ [\rho|_{\uU} \colon \uU \to \mathfrak{M}]
\end{align}
where $\yY \subset \xX$ is a closed substack and 
$\uU\cneq \xX \setminus \yY$. 
There is an associative $\ast$-product
on $H(X)$
based on the Ringel-Hall algebras. 
Let $\mathfrak{Ex}$ be the 
stack of short exact sequences 
\begin{align*}
0 \to E_1 \to E_3 \to E_2 \to 0
\end{align*}
in $\Coh_c(X)$
and 
\begin{align*}
p_i \colon \mathfrak{Ex} \to 
\mathfrak{M}
\end{align*} the 
1-morphism sending 
$E_{\bullet}$ to $E_i$. 
The $\ast$-product on $H(X)$ 
is given by 
\begin{align*}
[\rho_1 \colon \xX_1 \to \mathfrak{M}]
\ast [\rho_2 \colon \xX_2 \to \mathfrak{M}]
=[\rho_3 \colon \xX_3 \to \mathfrak{M}]
\end{align*}
where 
\begin{align*}
(\xX_3, \rho_3=p_3 \circ (\rho_1', \rho_2'))
\end{align*}
 is given by
the following Cartesian diagram
\begin{align*}
\xymatrix{
\xX_3 \ar[r]^{\hspace{-5mm}(\rho_1', \rho_2')}\ar[d] \ar@{}[dr]|\square
& \mathfrak{Ex} \ar[d]^{(p_1, p_2)}
  \ar[r]^{p_3} &
\mathfrak{M} \\
\xX_1 \times \xX_2  \ar[r]^{(\rho_1, \rho_2)} 
& \mathfrak{M}^{\times 2}.
& }
\end{align*}
Let $\cl$ be the group homomorphism 
\begin{align}\label{cl}
\cl \colon K(\Coh_c(X)) \to H^{\ast}(S, \mathbb{Q})
\end{align}
defined in the following way: 
\begin{align}\label{cl2}
\cl(E)=(\rank (\pi_{\ast}E), c_1(\pi_{\ast}E), \ch_2(\pi_{\ast}E)). 
\end{align}
We denote by $\Lambda \subset H^{\ast}(S, \mathbb{Q})$
the image of 
$\cl$. We write an element 
$\gamma \in \Lambda$ 
as $(r, l, s)$ as in the RHS of (\ref{cl2}). 
For $\gamma \in \Lambda$, 
let 
\begin{align}\label{Mgamma}
\mathfrak{M}(\gamma) \subset \mathfrak{M}
\end{align}
be the substack of $E \in \Coh_c(X)$ with $\cl(E)=\gamma$.
The algebra $H(X)$ is $\Lambda$-graded
\begin{align*}
H(X)=\bigoplus_{\gamma \in \Lambda} H_{\gamma}(X)
\end{align*}
where $H_{\gamma}(X)$ is $\mathbb{Q}$-spanned by 
the symbols (\ref{symbol})
which factor through (\ref{Mgamma}). 

\subsection{Integration map}
Let $\chi$ be the pairing on $\Lambda$
given by 
\begin{align}\label{pairing}
\chi((r_1, l_1, s_1), (r_2, l_2, s_2))=
K_S(r_2 l_1 -r_1 l_2). 
\end{align}
Since $X$ is a non-compact Calabi-Yau 3-fold, 
the Serre duality 
and the Riemann-Roch theorem implies 
\begin{align*}
\chi(\cl(E_1), \cl(E_2))
=&\dim \Hom(E_1, E_2) - \dim \Ext^1(E_1, E_2) \\
&+ \dim \Ext^1(E_2, E_1)
-\dim \Hom(E_2, E_1)
\end{align*}
for $E_1, E_2 \in \Coh_{c}(X)$. 
Let $C(X)$ be the Lie algebra
\begin{align*}
C(X)=\bigoplus_{\gamma \in \Lambda}
\mathbb{Q} \cdot c_{\gamma}
\end{align*}
with bracket given by 
\begin{align}\label{bracket}
[c_{\gamma_1}, c_{\gamma_2}]=(-1)^{\chi(\gamma_1, \gamma_2)}
\chi(\gamma_1, \gamma_2) c_{\gamma_1+\gamma_2}. 
\end{align}
There is a Lie subalgebra 
\begin{align*}
H^{\rm{Lie}}(X) \subset H(X)
\end{align*}
consisting of virtual indecomposable 
objects (cf.~\cite[Section~5.2]{Joy2})
and a linear homomorphism (cf.~\cite[Theorem~5.12]{JS})
\begin{align}\label{Lie}
\Pi \colon H^{\rm{Lie}}(X) \to C(X)
\end{align}
such that if $\xX$ is a $\mathbb{C}^{\ast}$-gerb over an 
algebraic space $\xX'$, we have 
\begin{align*}
\Pi([\rho \colon \xX \to \mathfrak{M}(\gamma)])
=-\left(\sum_{k\in \mathbb{Z}}
k \cdot \chi(\nu^{-1}(k))  \right) c_{\gamma}.
\end{align*}
Here $\nu$ is  
Behrend's constructible function~\cite{Beh}
on $\xX'$. 
Moreover the map (\ref{Lie})
preserves the brackets 
for the elements $[\rho_i \colon \xX_i \to \mathfrak{M}]$
for $i=1, 2$
if $\mathfrak{M}$ is a smooth stack 
at $\rho_i(x)$ for any $x\in \xX_i$, $i=1, 2$. 
In the case we are interested in, this
condition is satisfied: 
\begin{lem}\label{rmk:smooth}
Suppose that $K_S \cdot H <0$. Then for 
any $H$-slope semistable $E \in \Coh_c(X)$, 
the stack $\mathfrak{M}$ is smooth 
at $[E]$. 
\end{lem}
\begin{proof}
By Lemma~\ref{rmk:standard}, 
we have $E \in \Coh(S)$. 
Since the stability is an open condition, the 
obstruction space 
of the deformation theory of $E$ lies in $\Ext_{S}^2(E, E)$. 
By the Serre duality, we have
\begin{align*}
\Ext_S^2(E, E) \cong \Hom(E, E\otimes \omega_S)^{\vee}
\end{align*}
which vanishes since $E$, $E\otimes \omega_S$
are $H$-slope semistable
and $\mu_H(E)>\mu_H(E\otimes \omega_S)$
by $K_S \cdot H<0$. 
\end{proof}

\begin{rmk}\label{rmk:CS}
By the argument of~\cite[Theorem~5.12]{JS}, 
the map (\ref{Lie}) is a Lie algebra homomorphism
if we know that $\mathfrak{M}$ is 
analytically locally written as a critical locus 
of a certain holomorphic function 
in the sense of~\cite[Theorem~5.3]{JS}. 
However, since our situation is a non-compact Calabi-Yau 3-fold, 
we are not able to use~\cite[Theorem~5.12]{JS} to 
conclude that (\ref{Lie}) is a Lie algebra homomorphism. 
\end{rmk}
\subsection{Generalized DT invariants} 
For $\gamma \in \Lambda$, 
let 
\begin{align}\label{stack}
\mathfrak{M}_{H}^{s(ss)}(\gamma) \subset \mathfrak{M}(\gamma)
\end{align}
be the substack of
$H$-slope (semi)stable 
sheaves $E \in \Coh_c(X)$ satisfying
$\cl(E)=\gamma$. 
The stack (\ref{stack}) determines the element
\begin{align}\label{deltaH}
\delta_{H}(\gamma)=[\mathfrak{M}_{H}^{ss}(\gamma) \subset 
\mathfrak{M}(\gamma)] \in H_{\gamma}(X). 
\end{align}
The above element also defines the element of $H_{\gamma}(X)$\footnote{It 
is straightforward to check that (\ref{epH}) is a finite sum.}: 
\begin{align}\label{epH}
\epsilon_{H}(\gamma)
=\sum_{\begin{subarray}{c}
\gamma_1+ \cdots + \gamma_m=\gamma \\
\mu_H(\gamma_i)=\mu_H(\gamma)
\end{subarray}}
\frac{(-1)^{m-1}}{m}
\delta_{H}(\gamma_1) \ast \cdots \ast \delta_{H}(\gamma_m).
\end{align}
Here the slope $\mu_H(\gamma)$
for non-zero $\gamma=(r, l, s) \in \Lambda$ 
 is given by 
$l \cdot H/r$, i.e.
\begin{align*}
\mu_H(\cl(E))=\mu_H(E)
\end{align*}
holds 
for any non-zero $E \in \Coh_c(X)$. 
\begin{defi}\label{defi:genDT}
The generalized DT invariant
$\DT_{H}(\gamma) \in \mathbb{Q}$
is defined
by the formula:
\begin{align}\label{Pi:DT}
\Pi(\epsilon_{H}(\gamma))=-\DT_{H}(\gamma) \cdot c_{\gamma}. 
\end{align}
\end{defi}
\begin{rmk}\label{rmk:s=ss}
If $\mathfrak{M}_{H}^{s}(\gamma)=\mathfrak{M}_{H}^{ss}(\gamma)$, 
then they are $\mathbb{C}^{\ast}$-gerb over a
quasi-projective scheme $M_{H}^{s}(\gamma)$. 
In this case, the invariant $\DT_{H}(\gamma)$
is written as
\begin{align*}
\DT_{H}(\gamma)=\int_{M_{H}^{s}(\gamma)} \nu \ 
d\chi
\end{align*}
where $\nu$ is the Behrend function~\cite{Beh} on $M_H^{s}(\gamma)$. 
\end{rmk}
By formally replacing the Behrend function by the
constant function $1$ in the
construction of (\ref{Lie}), 
and removing the minus sign in (\ref{Pi:DT}), 
we obtain another invariant (cf.~\cite{Joy4}): 
\begin{align}\label{inv:Eu}
\Eu_{H}(\gamma) \in \mathbb{Q}. 
\end{align}
In the situation of Remark~\ref{rmk:s=ss}, the above 
invariant is the usual Euler number:
\begin{align*}
\Eu_{H}(\gamma)=\chi(M_{H}^{s}(\gamma)). 
\end{align*}
Also in the same situation of Lemma~\ref{rmk:standard}, the invariant
 (\ref{inv:Eu}) essentially 
coincides with the generalized DT invariant: 
\begin{lem}
Suppose that $K_S \cdot H<0$. Then we have 
the equality: 
\begin{align}\label{DT=Eu}
\DT_{H}(\gamma)=(-1)^{r^2 \chi(\oO_S)+1 + \Delta(\gamma)}
\Eu_{H}(\gamma). 
\end{align}
Here for $\gamma=(r, l, s) \in \Lambda$, 
the discriminant $\Delta(\gamma)$ is defined to be
\begin{align*}
\Delta(\gamma)=l^2 -2rs. 
\end{align*}
\end{lem}
\begin{proof}
For any closed point $[E] \in \mathfrak{M}_{H}^{ss}(\gamma)$, 
the stack $\mathfrak{M}$ is smooth at $[E]$ by Lemma~\ref{rmk:s=ss}. 
Its dimension is  
\begin{align*}
\dim \Ext_S^1(E, E)-\dim \Hom_S(E, E) =r^2 \chi(\oO_S)-\Delta(\gamma)
\end{align*}
by the Riemann-Roch theorem. 
Hence the Behrend function of $\mathfrak{M}$ at $[E]$
is given by $(-1)^{r^2 \chi(\oO_S)-\Delta(\gamma)}$. 
Taking the minus sign in (\ref{Pi:DT}) into account, 
we obtain the desired equality. 
\end{proof}
For $r\in \mathbb{Z}_{\ge 1}$ and $l \in \NS(S)$, 
we set
\begin{align}\label{DTHrl}
\DT_H(r, l) \cneq \sum_{s} \DT_H(r, l, s) (-q^{\frac{1}{2r}})^{l^2 -2rs}. 
\end{align}
Note that if $K_S \cdot H<0$, the equality (\ref{DT=Eu})
implies
\begin{align}\label{DTEuH}
\DT_H(r, l)=(-1)^{r^2 \chi(\oO_S)+1} \sum_{s} 
\Eu_H(r, l, s) q^{\frac{l^2}{2r}-s}. 
\end{align}
If furthermore $r=1$, 
the moduli stack $\mathfrak{M}_H^{ss}(1, l, s)$
is isomorphic to the $\mathbb{C}^{\ast}$-gerb over the 
Hilbert scheme of points on $S$. 
Hence we have (cf.~\cite{Got}):
\begin{align}\label{DT1l}
\DT(1, l)=(-1)^{\chi(\oO_S)+1} q^{\frac{\chi(S)}{24}} \eta(q)^{-\chi(S)}.  
\end{align}
Here $\eta(q)$ is the Dedekind eta function (\ref{eta}). 
Also our definition of the generating series (\ref{DTHrl})
implies 
that $\DT_H(r, l)$ depends on $l$ 
only on modulo $r$: 
\begin{lem}\label{lem:modr}
For any $l' \in \mathrm{NS}(S)$, we have
\begin{align*}
\DT_H(r, l+rl')=\DT_H(r, l). 
\end{align*}
\end{lem}
\begin{proof}
Let us take $L' \in \Pic(X)$
such that $c_1(L'|_{S})=l'$. 
The lemma follows since
$E \mapsto E \otimes L'$ preserves the 
$\mu_H$-semistability and 
$\Delta(E\otimes L')=\Delta(E)$. 
\end{proof}

\subsection{Generalized DT invariants for Gieseker semistable sheaves}
\label{lem:mu=G}
One may also be interested in generalized DT invariants 
counting Gieseker semistable sheaves. 
Indeed in the situation we are interested in (i.e. 
$S=\mathbb{P}^2$), they coincide with
the invariants $\DT_H(\gamma)$ in Definition~\ref{defi:genDT}. 
Let
\begin{align*}
\mathfrak{M}_{G, H}^{ss}(\gamma) \subset \mathfrak{M}
\end{align*}
be the substack of Gieseker semistable sheaves
$E \in \Coh_c(X)$ satisfying $\cl(E)=\gamma$. 
Similarly to (\ref{deltaH}), (\ref{epH}), 
we have the following elements of $H_{\gamma}(X)$: 
\begin{align}\notag
&\delta_{G, H}(\gamma)=[\mathfrak{M}_{G, H}^{ss}(\gamma) \subset 
\mathfrak{M}(\gamma)] \\
\label{ep:GH}
&\epsilon_{G, H}(\gamma)
=\sum_{\begin{subarray}{c}
\gamma_1+ \cdots + \gamma_m=\gamma \\
\overline{\chi}_H(\gamma_i)=\overline{\chi}_H(\gamma)
\end{subarray}}
\frac{(-1)^{m-1}}{m}
\delta_{G, H}(\gamma_1) \ast \cdots \ast \delta_{G, H}(\gamma_m). 
\end{align}
Here $\overline{\chi}_H(\gamma)$ for $\gamma \in \Lambda$
is determined by the condition 
$\overline{\chi}_H(\cl(E))=\overline{\chi}_H(E)$
for any $E \in \Coh_c(X)$. 
Similarly to Definition~\ref{defi:genDT}, we 
can define the invariant
\begin{align*}
\Pi(\epsilon_{G, H}(\gamma))=-\DT_{G, H}(\gamma) \cdot c_{\gamma}. 
\end{align*}

\begin{lem}\label{lem:compare}
Suppose that $-K_S$ is ample and 
$H=-aK_S$ for $a \in \mathbb{R}_{>0}$. 
Then we have the equality
\begin{align}\label{DT:equal}
\DT_{H}(\gamma)=\DT_{G, H}(\gamma). 
\end{align}
\end{lem}
\begin{proof}
By (\ref{obvious}), the argument of~\cite[Theorem~5.11]{Joy4} shows 
that
\begin{align}\label{equal:delta}
\delta_{H}(\gamma)
=\sum_{\begin{subarray}{c}
\gamma_1+ \cdots + \gamma_m=\gamma  \\
\mu_H(\gamma_i)=\mu_H(\gamma) \\
\overline{\chi}_H(\gamma_1) \succ \cdots \succ
\overline{\chi}_H(\gamma_m)
\end{subarray}}
\delta_{G, H}(\gamma_1) \ast \cdots \ast \delta_{G, H}(\gamma_m). 
\end{align}
By substituting (\ref{equal:delta}) into the RHS of 
(\ref{epH}), 
and using the inversion formula of (\ref{ep:GH})
as in~\cite[Equation~(23)]{Joy4}), 
we can describe $\epsilon_H(\gamma)$ in terms of 
$\epsilon_{G, H}(\gamma_i)$
with $\mu_H(\gamma_i)=\mu_H(\gamma)$. 
Using~\cite[Theorem~5.4]{Joy4}, the same
argument of~\cite[Theorem~5.2]{Joy4}
shows that
$\epsilon_{H}(\gamma)$ is written as
\begin{align}\label{commutator}
\epsilon_{H}(\gamma)
=\epsilon_{G, H}(\gamma)+
\left( \begin{array}{c}
\mbox{ multiple commutators of } \\
\epsilon_{G, H}(\gamma_i) 
\mbox{ with }
\mu_H(\gamma_i)=\mu_H(\gamma) 
\end{array} \right). 
\end{align}
By our assumption $H=-aK_S$, 
(\ref{pairing}), (\ref{bracket}), and Lemma~\ref{rmk:smooth}, 
we have
\begin{align*}
\Pi [\epsilon_{G, H}(\gamma_1), \epsilon_{G, H}(\gamma_2)]=0, \ 
\mbox{ if } \mu_H(\gamma_i)=\mu_H(\gamma). 
\end{align*}
Applying $\Pi$ to (\ref{commutator}), 
we obtain the desired equality (\ref{DT:equal}). 
\end{proof}

\subsection{Wall-crossing formula}\label{subsec:wcf}
The behavior of $\DT_H(\gamma)$, $\Eu_H(\gamma)$
under the change of $H$ is described 
by the wall-crossing formula given in~\cite{Joy4}, \cite{JS}. 
Here we recall its explicit formula
for $\Eu_H(\gamma)$. 
Let $H_1$, $H_2$ be $\mathbb{R}$-divisors on $S$. 
We recall some combinatorial numbers:
\begin{defi}\emph{(\cite[Definition~4.2]{Joy4})}
For non-zero $\gamma_1, \cdots, \gamma_m \in \Lambda$, 
we define 
\begin{align*}
S(\{\gamma_1, \cdots, \gamma_m\}, H_1, H_2) \in \{0, \pm 1\}
\end{align*}
as follows: 
if for each $i=1, \cdots, m-1$, we have either 
(\ref{S1}) or (\ref{S2})
\begin{align}
\label{S1}
&\mu_{H_1}(\gamma_i) \le \mu_{H_1}(\gamma_{i+1}) \mbox{ and }
\mu_{H_2}(\gamma_1+ \cdots + \gamma_i)>
\mu_{H_2}(\gamma_{i+1}+ \cdots + \gamma_{m}) \\
\label{S2}
&\mu_{H_1}(\gamma_i) > \mu_{H_1}(\gamma_{i+1}) \mbox{ and }
\mu_{H_2}(\gamma_1+ \cdots + \gamma_i) \le
\mu_{H_2}(\gamma_{i+1}+ \cdots + \gamma_{m})
\end{align}
then define 
\begin{align*}
S(\{\gamma_1, \cdots, \gamma_m\}, H_1, H_2)=(-1)^{k}
\end{align*}
where $k$ is the number of 
$i=1, \cdots, m-1$ satisfying (\ref{S1}). 
Otherwise we define $S(\{\gamma_1, \cdots, \gamma_m\}, H_1, H_2)=0$.
\end{defi}
Another combinatorial number is defined as follows: 
\begin{defi}\emph{(\cite[Definition~4.4]{Joy4})}
For non-zero $\gamma_1, \cdots, \gamma_m \in \Lambda$, 
we define 
\begin{align}\notag
U(\{\gamma_1, \cdots, \gamma_m\}, H_1, H_2)
=\sum_{1\le m'' \le m' \le m}
\sum_{\begin{subarray}{c}
\psi \colon \{1, \cdots, m\} \to \{1, \cdots, m'\} \\
\label{defi:U}
\psi' \colon \{1, \cdots, m'\} \to \{1, \cdots, m''\}
\end{subarray}} \\
\prod_{a=1}^{m''}S(\{\Upsilon_i\}_{i \in \psi^{'-1}(a)}, H_1, H_2)
\frac{(-1)^{m^{''}-1}}{m''}
\prod_{b=1}^{m'} \frac{1}{\lvert \psi^{-1}(b) \vert !}. 
\end{align}
Here $\psi$, $\psi'$, $\Upsilon_i$ are as follows: 
\begin{itemize}
\item $\psi$ and $\psi'$ are non-decreasing surjective maps. 
\item For $1\le i, j \le m$ with $\psi(i)=\psi(j)$, 
we have 
$\mu_{H_1}(\gamma_i)=\mu_{H_1}(\gamma_j)$. 
\item For $1\le i, j \le m''$, we have 
\begin{align}\label{m''}
\mu_{H_2}\left( \sum_{k \in \psi^{-1} \psi^{'-1}(i)} \gamma_k \right)
=\mu_{H_2}\left( \sum_{k \in \psi^{-1} \psi^{'-1}(j)} \gamma_k \right). 
\end{align}
\item The elements $\Upsilon_i \in \Lambda$ 
for $1\le i\le m'$ are defined to be
\begin{align}\label{wi}
\Upsilon_i=\sum_{j \in \psi^{-1}(i)} \gamma_j. 
\end{align}
\end{itemize}
\end{defi}
For $m\in \mathbb{Z}_{\ge 1}$, let
$G(m)$ be the set of 
connected, simply connected graphs
with vertex $\{1, \cdots, m\}$, 
such that $i \to j$ in $G$ implies $i<j$.  
The wall-crossing formula for $\Eu_{H}(\gamma)$ is described in
the following way: 
\begin{thm}\emph{(\cite[Theorem~6.28, Equation~(130)]{Joy4})}\label{thm:Joy4}
Suppose that $H_1, H_2$ are ample 
$\mathbb{R}$-divisors. We have the formula: 
\begin{align}\notag
\Eu_{H_2}(\gamma)=
\sum_{\begin{subarray}{c}
m\ge 1, \gamma_1, \cdots, \gamma_m \in \Lambda \\
\gamma_1+ \cdots + \gamma_m=\gamma
\end{subarray}}
\sum_{G \in G(m)}
\frac{1}{2^{m-1}}U(\{\gamma_1, \cdots, \gamma_m\}, H_1, H_2)
\\ \label{WCF}
\prod_{i \to j \text{ \rm{in} }G} \chi(\gamma_i, \gamma_j) \prod_{i=1}^{m} \Eu_{H_1}(\gamma_i). 
\end{align}
\end{thm}
\begin{rmk}
If we know that the stack $\mathfrak{M}$
satisfies the property as in Remark~\ref{rmk:CS}, 
we can apply (\ref{Lie}) to show the wall-crossing formula 
for $\DT_H(\gamma)$ similar to (\ref{WCF}) 
as in~\cite[Theorem~6.28]{JS}. 
Alternatively, if $K_S \cdot H_i<0$, we 
can substitute the equality (\ref{DT=Eu}) to (\ref{WCF})
and obtain the wall-crossing formula for $\DT_H(\gamma)$. 
\end{rmk}

\subsection{Theta type series for indefinite lattices}\label{subsec:theta}
We introduce the theta type series 
from data
\begin{align}\label{xi}
\xi=(\Gamma, B, \overline{\nu}, c_1, c_2, \cdots, c_b, c_1', 
c_2', \cdots, c_b', \alpha_1, \cdots, \alpha_k)
\end{align}
satisfying the following conditions: 
\begin{itemize}
\item (i) $\Gamma$ is a finitely 
generated free abelian group and 
\begin{align*}
B \colon \Gamma \times \Gamma \to \mathbb{Z}
\end{align*}
a non-degenerate 
symmetric bilinear pairing
with index $(a, b)$ for $a\ge b$. 

\item (ii) The elements  
$c_1, \cdots, c_b \in \Gamma_{\mathbb{Q}}$ 
span a $b$-dimensional negative 
definite subspace in $\Gamma_{\mathbb{Q}}$. 

\item (iii) The elements $c_1', \cdots, c_b' \in \Gamma_{\mathbb{Q}}$
satisfy 
\begin{align*}
&B(c_i, c_j')=0 \ \mbox{ for all }
1\le i, j \le b, \ i\neq j \\ 
&B(c_i', c_j')=0 \ \mbox{ for all } 
1\le i, j \le b \\
&B(c_i, c_i')<0 \ \mbox{ for all }
1\le i\le b. 
\end{align*}

\item (iv) The element $\overline{\nu} \in \Gamma_{\mathbb{Q}}$
satisfies that
\begin{align*}
B(c_i', \nu) \neq 0 \ 
\mbox{ for all } 1\le i\le b \mbox{ and } \nu \in \overline{\nu} +\Gamma. 
\end{align*}

\item (v) $k\in \mathbb{Z}_{\ge 0}$
and $\alpha_1, \cdots, \alpha_k$
are elements of $\Gamma_{\mathbb{Q}}$. 
\end{itemize}
As in the introduction, we 
set $Q(\nu)=B(\nu, \nu)/2$ and 
consider the series
\begin{align}\label{Theta}
&\Theta_{\xi}(q) 
\cneq \\
\notag
&\sum_{\nu \in \overline{\nu} + \Gamma}
\prod_{i=1}^{b}\left( \sgn \left(B(c_i, \nu)\right)-
\sgn \left( B(c_i', \nu) \right) \right)
\prod_{j=1}^{k} B(\alpha_j, \nu) \cdot q^{Q(\nu)}. 
\end{align}
When $b=k=0$, 
the series (\ref{Theta})
becomes
\begin{align}\label{Theta:classical}
\Theta_{\xi}(q)=\sum_{\nu \in \overline{\nu}+\Gamma} q^{Q(\nu)}. 
\end{align}
The series (\ref{Theta:classical}) 
is a classical theta series with respect to the positive 
definite quadratic form $Q$ on $\Gamma$, 
which 
is a modular form with weight $a/2$.
If $b=0$ and $k>0$, then the series (\ref{Theta}) is obtained 
as derivations of Jacobi theta series with respect to 
elliptic variables. 

If $b=1$ and $k=0$, 
then the series (\ref{Theta})
is not always a modular form. 
Instead 
Zwegers~\cite{Zweg}
showed that the series
\begin{align*} 
\sum_{\nu \in \overline{\nu} + \Gamma}
 \left( E\left( \frac{B(c_1, \nu)}{\sqrt{-Q(c_1)}} 
y^{\frac{1}{2}} \right)-
\sgn (B(c_1', \nu)) \right)
 e^{2\pi i Q(\nu) \tau}
\end{align*}
for $y=\Imm \tau$ gives a real analytic modular form 
of weight $(a+1)/2$. 
Here $E(x)$ is defined by 
\begin{align*}
E(x)=2 \int_{0}^{x} e^{-\pi u^2} du. 
\end{align*}
A series which admits 
a modular completion as above is called 
a mock modular form~\cite{Zagier}.   
The case $b=1$ and $k>0$ is obtained by the 
derivations of mock Jacobi forms in~\cite{Zweg}
with respect to the elliptic variables. 

Suppose that $b\ge 2$ and $k=0$. 
If we further assume that 
\begin{align}\label{further}
B(c_i, c_j)=0, \ i\neq j
\end{align}
then the argument of~\cite{Zweg} can 
be easily generalized to show 
that 
\begin{align}\label{bmock}
\sum_{\nu \in \overline{\nu} + \Gamma}
\prod_{i=1}^{b}
 \left( E\left( \frac{B(c_i, \nu)}{\sqrt{-Q(c_i)}} 
y^{\frac{1}{2}} \right)-
\sgn (B(c_i', \nu)) \right)
 e^{2\pi i Q(\nu) \tau}
\end{align}
is a real analytic modular form. 
Indeed the series (\ref{Theta}) in this case
is a mixed mock modular form in the sense of~\cite{BrMan}. 
\begin{rmk}
Unfortunately
 the series (\ref{Theta}) without 
the condition (\ref{further}) 
is involved in Theorem~\ref{intro:main}. 
In that case, 
the proof of~\cite{Zweg} 
is not directly applied to show the 
modularity of (\ref{bmock}). 
The study of the modularity of (\ref{bmock}) without (\ref{further}),  
or other kind of modular completion of the series (\ref{Theta}), 
would be required to 
understand the S-duality for an arbitrary rank. 
\end{rmk}

\subsection{Some properties of theta type series}
Let us 
consider the series (\ref{Theta}) determined by data 
(\ref{xi}) satisfying (i) to (v). 
We first show the convergence of (\ref{Theta}): 
\begin{lem}\label{lem:converge}
For $\tau \in \hH$, the series
\begin{align}\label{Theta:tau}
\sum_{\nu \in \overline{\nu} + \Gamma}
\prod_{i=1}^{b} \left( \sgn (B(c_i, \nu))-
\sgn ( B(c_i', \nu) ) \right)
\prod_{j=1}^{k} B(\alpha_j, \nu) \cdot e^{2\pi i Q(\nu)\tau}
\end{align}
converges absolutely. 
\end{lem}
\begin{proof}
The convergence for $(b, k)=(1, 0)$ 
follows from~\cite[Proposition~2.4]{Zweg}. 
The conditions (i) to (v) allow us
to apply a similar argument. 
Since the series (\ref{Theta}) is unchanged 
by replacing $c_i, c_i'$ by multiplications
of positive integers, we may assume that $c_i, c_i' \in \Gamma$. 
By the condition (iii) of data (\ref{xi}), 
the element $\nu \in \overline{\nu} + \Gamma$
is uniquely written as 
\begin{align*}
\nu=\mu+ \sum_{i=1}^{b}m_i c_i'
\end{align*}
for some $\mu \in \overline{\nu} +\Gamma$, 
$m_i \in \mathbb{Z}$ satisfying 
\begin{align*}
\frac{B(c_i, \mu)}{B(c_i, c_i')} \in [0, 1) \ 
\mbox{ for all } 1\le i\le b.  
\end{align*}
Therefore the series (\ref{Theta:tau}) is written as
\begin{align}\notag
\sum_{\begin{subarray}{c}
\mu \in \overline{\nu} + \Gamma \\
\frac{B(c_i, \mu)}{B(c_i, c_i')} \in [0, 1) \\
\text{for all }
1\le i\le b
\end{subarray}}
e^{2\pi i Q(\mu)\tau}
&\sum_{m_1, \cdots, m_b \in \mathbb{Z}}
\prod_{i=1}^{b}
\left( \sgn \left( \frac{B(c_i, \mu)}{B(c_i, c_i')} +m_i \right)  
+\sgn(B(c_i', \mu)) \right) \\
\label{theta:1}
&\cdot (-1)^b \prod_{i=1}^{k} B\left(\alpha_i, \mu+\sum_{i=1}^{b}m_i c_i'
\right)
e^{\sum_{j=1}^{b} 2\pi i B(c_j', \mu) m_j \tau}. 
\end{align}
Here we have used $B(c_i, c_i')<0$ 
from the condition (ii). 
Since there is a finite number
of possibilities for the value
$B(c_i, \mu)/B(c_i, c_i') \in [0, 1)$
in (\ref{theta:1}), 
there is a finite number of $\mu_1, \cdots, \mu_p \in \Gamma$
such that 
any $\mu \in \overline{\nu} + \Gamma$ 
in (\ref{theta:1})
is written 
as 
\begin{align*}
\mu= \overline{\nu} + \mu_e + \mu'
\end{align*}
for some $1\le e\le p$ and 
$\mu' \in \Gamma'$, 
where $\Gamma' \subset \Gamma$ is the orthogonal 
complement of the $(0, b)$-space spanned by 
$c_1, \cdots, c_b$. 
Therefore the series (\ref{theta:1})
is a finite linear combination
 of the series of the form
\begin{align*}
&\sum_{\mu \in \overline{\nu}+\mu_e+\Gamma'}
\prod_{s\in S}B(\alpha_s, \mu) e^{2\pi i Q(\mu) \tau} \\
&\cdot
\prod_{j=1}^{b} \sum_{m_j \in \mathbb{Z}}
m_j^{k_j}\left( \sgn \left( \frac{B(c_j, \mu)}{B(c_j, c_j')} +
m_j \right)  
+\sgn(B(c_j', \mu)) \right) 
e^{2\pi i B(c_j', \mu) m_j \tau}
\end{align*}
for some fixed $e \in \{1, \cdots, p\}$, 
a finite set $S \subset \{1, \cdots, k\}$
and 
some $k_1, \cdots, k_b \in \mathbb{Z}_{\ge 0}$. 
Let us also fix the elements 
\begin{align*}
\mu \in \overline{\nu} + \mu_e +\Gamma', \
j \in \{1, \cdots, b\}.
\end{align*}
Then we have 
\begin{align*}
&\sum_{m_j \in \mathbb{Z}}
m_j^{k_j}\left( \sgn \left( \frac{B(c_j, \mu)}{B(c_j, c_j')} +
m_j \right)  
+\sgn(B(c_j', \mu)) \right) 
e^{2\pi i B(c_j', \mu) m_j \tau} \\
&=2 \sgn(B(c_j', \mu)) \sum_{m_j \in \mathbb{Z}_{\ge 0}}m_j^{k_j}
e^{2\pi i \lvert B(c_j', \mu) \rvert m_j \tau}
+C \\
&=2\sgn(B(c_j', \mu))(2\pi i \lvert B(c_j', \mu) \rvert)^{-k_j}
\left( \frac{d}{d\tau} \right)^{k_j}
\frac{1}{1-e^{2\pi i \lvert B(c_j', \mu)\rvert \tau}}
+C. 
\end{align*}
Here $C\in \{0, \pm 1\}$, depending on the signs 
of $B(c_j', \mu)$ and $B(c_j, \mu)/B(c_j, c_j')$. 
Since we have
\begin{align*}
\mathrm{inf} \{ \lvert B(c_j', \mu) \rvert : \mu \in \overline{\nu} +\Gamma\}>0
\end{align*}
due to the condition (iv),
we have 
\begin{align*}
\mathrm{sup} \left\{ \left\lvert B(c_j', \mu)^{-k_j}
\left( \frac{d}{d\tau} \right)^{k_j}
\frac{1}{1-e^{2\pi i \lvert B(c_j', \mu)\rvert \tau}} \right\rvert :
\mu \in \overline{\nu} + \Gamma \right\} <\infty. 
\end{align*}
We are reduced to showing the absolute convergence of 
the series
\begin{align}\label{theta:2}
\sum_{\mu \in \overline{\nu}+\mu_e+\Gamma'}
\prod_{s\in S}B(\alpha_s, \mu) e^{2\pi i Q(\mu) \tau}. 
\end{align}
Since $\Gamma'$ is positive definite by the condition (ii), 
the series (\ref{theta:2}) converges absolutely 
by the absolute convergence of the classical theta series. 
\end{proof}
By the proof of Lemma~\ref{lem:converge},
the series $\Theta_{\xi}(q)$ in (\ref{Theta})
makes sense, and determines the element 
\begin{align*}
\Theta_{\xi}(q) \in \mathbb{Q}(\hspace{-1mm}( q^{\frac{1}{N}})\hspace{-1mm})
\end{align*} 
for some $N\in \mathbb{Z}_{\ge 1}$. 
\begin{defi}\label{defi:subM}
We define  
\begin{align}\label{subM}
\mM \subset \lim_{\longrightarrow}
\mathbb{Q}(\hspace{-1mm}( q^{\frac{1}{N}})\hspace{-1mm})
\end{align}
to
be the $\mathbb{Q}$-subalgebra
generated by $\Theta_{\xi}(q^{\frac{1}{N}})$ for all the data (\ref{xi}) 
and $N\in \mathbb{Z}_{\ge 1}$. 
\end{defi}
We will use the following lemma: 
\begin{lem}\label{lem:classical}
Let $\xi$ be data (\ref{xi}), 
$V \subset \Gamma_{\mathbb{Q}}$ a linear subspace
which contains $c_1, \cdots, c_b$, $c_1', \cdots, c_b'$, 
and $T \subset \{1, \cdots, b\}$ a subset. 
Then the series
\begin{align}\label{series:V}
\sum_{\begin{subarray}{c}
\nu \in (\overline{\nu} + \Gamma) \cap V \\
B(\nu, c_i)=0 \text{ \rm{for all} } i \in T
\end{subarray}}
\prod_{i\notin T} \left( \sgn (B(c_i, \nu))-
\sgn ( B(c_i', \nu) ) \right)
\prod_{j=1}^{k}B(\alpha_j, \nu) q^{Q(\nu)} 
\end{align}
is an element of $\mM$. 
\end{lem}
\begin{proof}
If the series (\ref{series:V}) is not zero, there is 
$\nu_0 \in (\overline{\nu}+\Gamma) \cap V$
with $B(\nu_0, c_i)=0$ for all $i \in T$. 
Any element $\nu$ in the series (\ref{series:V})
satisfies that
\begin{align*}
\nu-\nu_0 \in \Gamma' \cneq 
\{ v \in \Gamma \cap V : B(v, c_i) =0 \text{ for all } i \in T\}. 
\end{align*}
Note that $c_i'$ for $i\notin T$ and 
$\nu_0$ are elements of $\Gamma'_{\mathbb{Q}}$. 
We consider the decomposition
\begin{align*}
\Gamma_{\mathbb{Q}}=\Gamma'_{\mathbb{Q}} \oplus 
\Gamma_{\mathbb{Q}}^{'\perp}
\end{align*}
where $\Gamma_{\mathbb{Q}}^{'\perp}$ is the orthogonal 
complement of $\Gamma'_{\mathbb{Q}}$ in $\Gamma_{\mathbb{Q}}$
with respect to $B(-, -)$. 
For $\nu \in \Gamma_{\mathbb{Q}}$, we denote by $\nu^{+}$ the
$\Gamma'_{\mathbb{Q}}$-component 
of $\nu$ with respect to the above decomposition. 
The series (\ref{series:V}) is written as
\begin{align}\label{series:dash}
\sum_{\begin{subarray}{c}
\nu \in \nu_0 + \Gamma'
\end{subarray}}
\prod_{i\notin T} 
\left( \sgn (B(c_i^{+}, \nu))-
\sgn ( B(c_i', \nu) ) \right)
 \prod_{j=1}^{k}B(\alpha_j^{+}, \nu) \cdot q^{Q(\nu)}. 
\end{align}
Since $c_1, \cdots, c_b$ span 
a $(0, b)$-space, the elements 
$c_i$ with $i \in T$ span a $(0, \lvert T \rvert)$-space
and $c_i^{+}$ with $i\notin T$ span a 
$(0, b-\lvert T \rvert)$-space. 
Hence the data
\begin{align*}
(\Gamma', B|_{\Gamma'}, \nu_0, c_i^{+}, c_i', i\notin T, \alpha_1^{+}, \cdots, 
\alpha_k^{+})
\end{align*}
satisfies the conditions (i) to (v) in the previous subsection. 
Therefore the series (\ref{series:dash}) is an element of $\mM$. 
\end{proof}
\begin{lem}\label{lem:element}
Any element in $\mM$ is written as
\begin{align*}
\sum_{i=1}^{n} a_i \Theta_{\xi_i}(q^{\frac{1}{N}})
\end{align*}
for a finite number of data $\xi_1, \cdots, \xi_n$
as in (\ref{xi}), $a_1, \cdots, a_n \in \mathbb{Q}$
and $N\in \mathbb{Z}_{\ge 1}$.  
\end{lem}
\begin{proof} 
If $\xi_1, \xi_2$ are data (\ref{xi}), then 
we have 
\begin{align*}
\Theta_{\xi_1}(q) \cdot \Theta_{\xi_2}(q) =\Theta_{\xi_1 \oplus \xi_2}(q)
\end{align*} 
where $\xi_1 \oplus \xi_2$ is the direct product of 
data $\xi_1$, $\xi_2$ in an obvious sense. 
Moreover $\Theta_{\xi}(q)=\Theta_{\xi'}(q^{1/N})$ where
$\xi'$ is data (\ref{xi}) with $B$ replaced by 
$NB$. Therefore the lemma holds. 
\end{proof}

\section{Proof of Theorem~\ref{intro:main}}
In this section, we prove Theorem~\ref{intro:main}. Below, 
we denote by $H$ the hyperplane class of $\mathbb{P}^2$. 
We identify $\NS(\mathbb{P}^2)$ with 
$\mathbb{Z}$ by $lH \mapsto l$.
For $r\in \mathbb{Z}_{\ge 1}$ and $l\in \mathbb{Z}$, 
we consider the generating 
series 
\begin{align*}
\DT(r, l) \cneq \DT_H(r, l) 
\end{align*} 
defined by (\ref{DTHrl}) for $S=\mathbb{P}^2$. 
By the Bogomolov inequality, the above
series coincides with the series (\ref{intro:gen}) in
the introduction.

\subsection{Blow-up formula}
Let 
\begin{align*}
f \colon \widehat{\mathbb{P}}^2 \to \mathbb{P}^2
\end{align*}
be a blow-up at a point in $\mathbb{P}^2$. 
Note that we have
\begin{align*}
\NS(\widehat{\mathbb{P}}^2)=\mathbb{Z}[f^{\ast}H] \oplus 
\mathbb{Z}[C].
\end{align*}
Below we write an element $l f^{\ast}H +aC$ 
of $\NS(\wP)$ 
as $(l, a)$. 
We have the following blow-up formula of the series (\ref{DTHrl}):
\begin{prop}\label{prop:blow}
For any $r\in \mathbb{Z}_{\ge 1}$, $l\in \mathbb{Z}$
and $a \in \mathbb{Z}$, 
we have the following formula:
\begin{align}\label{blow:DT}
\DT_{H_0}(r, (l, a))=q^{\frac{r}{24}}\eta(q)^{-r}
\cdot \vartheta_{r, a}(q) \cdot \DT(r, l). 
\end{align}
Here $H_0=f^{\ast}H$ and $\vartheta_{r, a}(q)$ is 
defined by 
\begin{align}\notag
\vartheta_{r, a}(q) \cneq \sum_{\begin{subarray}{c}
(k_1, \cdots, k_{r-1}) \in  (a/r, \cdots, a/r) + \mathbb{Z}^{r-1}
\end{subarray}}
q^{\sum_{1\le i\le j\le r-1}k_i k_j}. 
\end{align}
\end{prop}
\begin{proof}
If $r$ and $l$ are coprime, the 
result essentially follows from~\cite{Yo1}, \cite{WZ}, \cite{GoTheta}. 
In a general case, 
we use the blow-up formula in~\cite{TodS} for the 
invariants $\Eu_H(\gamma)$ obtained as an application 
of the flop transformation formula of 
generalized DT type invariants. 
We note that, although 
$H_0$ is not ample, 
 the LHS of (\ref{blow:DT})
is well-defined due to the boundedness of $\mu_{H_0}$-semistable 
sheaves on $\widehat{\mathbb{P}}^2$ (cf.~\cite[Proposition~2.17]{TodS}). 
By~\cite[Theorem~4.3]{TodS}, we have 
\begin{align}\notag
\sum_{s, a} \Eu_{H_0}(r, (l, a), -s)
&q^{\frac{r}{12}-\frac{a}{2}+s}t^{\frac{r}{2}-a} \\
\label{Eu:blow}
&=\sum_{s} \Eu_{H}(r, l, -s)q^{s} 
\cdot \eta(q)^{-r} 
\vartheta_{1, 0}(q, t)^{r}. 
\end{align}
 Here $\eta(q)$ is given by (\ref{eta}) 
and $\vartheta_{1, 0}(q, t)$ is given by
\begin{align*}
\vartheta_{1, 0}(q, t)=\sum_{k\in \mathbb{Z}}
q^{\frac{1}{2}\left( k + \frac{1}{2} \right)^2 }t^{k+\frac{1}{2}}. 
\end{align*}
The formulas (\ref{DTEuH}) and (\ref{Eu:blow})
immediately imply 
\begin{align*}
&\DT_{H_0}(r, (l, a)) \\
&=q^{\frac{r}{24}}\eta(q)^{-r}
\cdot \left(\sum_{\begin{subarray}{c}
k_1, \cdots, k_r \in \mathbb{Z} \\
k_1 + \cdots + k_r=-a\end{subarray}}  
q^{\frac{1}{2}(k_1^2 + \cdots + k_r^2)-\frac{a^2}{2r}}
\right) \cdot \DT(r, l). 
\end{align*}
By the substitution $k_r=-a-k_1 - \cdots - k_{r-1}$, it is 
straightforward to check that
\begin{align*}
\sum_{\begin{subarray}{c}
k_1, \cdots, k_r \in \mathbb{Z} \\
k_1 + \cdots + k_r=-a\end{subarray}}  
q^{\frac{1}{2}(k_1^2 + \cdots + k_r^2)-\frac{a^2}{2r}}=\vartheta_{r, a}(q). 
\end{align*}
\end{proof}

\subsection{Combinatorial generating series}
In this subsection, we introduce 
some generating series defined by 
the combinatorial numbers in Subsection~\ref{subsec:wcf}. 
 For $t\in \mathbb{R}$, we 
set \begin{align*}
H_t \cneq f^{\ast}H -t C \in \NS(\widehat{\mathbb{P}}^2)_{\mathbb{R}}. 
\end{align*}
Note that $H_t$ is ample for $t \in (0, 1)$, 
$H_0=f^{\ast}H$ is nef and big, and 
\begin{align*}
F \cneq H_1 =f^{\ast}H -C
\end{align*}
is a fiber class of the $\mathbb{P}^1$-fibration
$\widehat{\mathbb{P}}^2 \to \mathbb{P}^1$. 
Also we denote by $\Lambda \subset H^{\ast}(\wP, \mathbb{Q})$
the image of $\cl$ in (\ref{cl}) for $S=\wP$. 
Let us take $m\in \mathbb{Z}_{\ge 1}$ and 
\begin{align*}
r_1, \cdots, r_m \in \mathbb{Z}_{\ge 1}, \ 
\beta_1, \cdots, \beta_m \in \NS(\wP).
\end{align*}
We set
\begin{align*}
&S(\{(r_i, \beta_i)\}_{i=1}^{m}, H, F_{+})
\cneq \lim_{t \to 1-0}
S(\{(r_i, \beta_i, 0)\}_{i=1}^{m}, H_0, H_t) \\
&U(\{(r_i, \beta_i)\}_{i=1}^{m}, H, F_{+})
\cneq \lim_{t \to 1-0}
U(\{(r_i, \beta_i, 0)\}_{i=1}^{m}, H_0, H_t). 
\end{align*}
Here we regard $(r_i, \beta_i, 0)$ as an element\footnote{The choice $0$ in the $H^4$-component can be arbitrary, since 
the slope is independent of the second Chern character.}
of $\Lambda$, and 
$S$, $U$ are combinatorial numbers
 in Subsection~\ref{subsec:wcf}. 
Obviously the limits of the RHS are well-defined. 

We introduce some more 
notation. For $r\ge 1$, we set
\begin{align*}
&\NS_{< r}(\wP) \\
&\cneq \{ x f^{\ast}H + yC : x, y \in \mathbb{Z}, \ 
0\le x \le r-1, 0\le y \le r-1 \}. 
\end{align*}
Note that $\NS_{< r}(\wP)$ is a finite subset of 
$\NS(\wP)$. 
Also we denote by $G'(m)$
the set of oriented graphs $G$ with 
vertex a subset in $\{1, \cdots, m\}$, which 
may not be connected nor simply connected, 
and $i \to j$ implies $i\le j$. 
Note that $G(m) \subset G'(m)$, where $G(m)$ is the 
set of graphs in Subsection~\ref{subsec:wcf}. 
\begin{defi}
Given data
\begin{align}\label{data:W}
&l \in \mathbb{Z}, \ m \in \mathbb{Z}_{\ge 1}, \ G \in G'(m) \\
\notag
&r_1, \cdots, r_m \in \mathbb{Z}_{\ge 1}, \ 
\overline{\beta}_i \in \NS_{< r_i}(\wP) \ (1\le i\le m)
\end{align}
we define
the following generating series
\begin{align}\notag
S_{(r_1, \overline{\beta}_1), \cdots, (r_m, \overline{\beta}_m)}^{l, G}(q)
\cneq \sum_{\begin{subarray}{c}
\beta_i \in \NS(\wP), \
\beta_i \equiv \overline{\beta}_i \ (\text{\rm{mod} } r_i) \\
\beta_1 + \cdots + \beta_m=(l, 1-l) 
\end{subarray}}
S(\{(r_i, \beta_i)\}_{i=1}^{m}, H, F_{+}) \\
\label{S(q)}
\cdot \prod_{i \to j \text{ \rm{in} } G} K_{\wP}(r_j \beta_i - r_i \beta_j) q^{-\sum_{1\le i<j \le m} \frac{(r_j \beta_i -r_i \beta_j)^2}{2r r_i r_j}} \\
\notag
U_{(r_1, \overline{\beta}_1), \cdots, (r_m, \overline{\beta}_m)}^{l, G}(q)
\cneq \sum_{\begin{subarray}{c}
\beta_i \in \NS(\wP), \
\beta_i \equiv \overline{\beta}_i \ (\text{\rm{mod} } r_i) \\
\beta_1 + \cdots + \beta_m=(l, 1-l)
\end{subarray}}
U(\{(r_i, \beta_i)\}_{i=1}^{m}, H, F_{+}) \\
\label{W(q)}
\cdot \prod_{i \to j \text{ \rm{in} } G} K_{\wP}(r_j \beta_i - r_i \beta_j) q^{-\sum_{1\le i<j \le m} \frac{(r_j \beta_i -r_i \beta_j)^2}{2r r_i r_j}}.
\end{align}
\end{defi}
The generating series (\ref{S(q)}), (\ref{W(q)}) are well-defined. 
Indeed, we have the following proposition: 
\begin{prop}\label{prop:SU}
The series (\ref{S(q)}), (\ref{W(q)}) are elements of $\mM$. 
\end{prop}
Here $\mM$ is given in Definition~\ref{defi:subM}. 
The proof of Proposition~\ref{prop:SU} will be
given in Subsection~\ref{subsec:propS} and Subsection~\ref{subsec:propU}.

\subsection{Rank reduction formula}
In this subsection, we apply Theorem~\ref{thm:Joy4}
to describe
$\DT(r, l)$ for $r\ge 2$ in terms of the 
series (\ref{W(q)}) and the series $\DT(r', l')$ with $r'<r$. 
We first collect some well-known lemmas: 
\begin{lem}\label{lem:wall}
For a fixed $\gamma=(r, \beta, s) \in \Lambda$, there exist 
\begin{align*}
0=t_0<t_1< \cdots < t_m=1
\end{align*}
 such that 
the stack $\mathfrak{M}^{ss}_{H_t}(\gamma)$ is constant 
on $t \in (t_i, t_{i+1})$. 
\end{lem}
\begin{proof}
It is enough to show that the set of $t \in [0, 1)$ 
satisfying the following: 
there exist $\mu_{H_t}$-semistable 
sheaves $E_i$ with $\cl(E_i)=\gamma_i=(r_i, \beta_i, s_i) \in \Lambda$ 
for $i=1, 2$
such that 
\begin{align*}
\mu_{H_t}(\gamma_i)=\mu_{H_t}(\gamma), \ 
\gamma_1 + \gamma_2=\gamma.
\end{align*}
By the left equality and the Hodge index theorem, we 
have 
\begin{align}\label{Hodge}
(r\beta_1 - r_1 \beta)^2 \le 0. 
\end{align} 
On the other hand,
we have $s_i \le \beta_i^2/2r_i$
by Bogomolov inequality, hence  
$\beta^2_1/r_1 + \beta_2^2/r_2 \ge 2s$.  
By substituting $\beta_2=\beta-\beta_1$,
we obtain 
\begin{align}\label{below}
-r_1 r_2 \Delta(\gamma) \le (r\beta_1 - r_1 \beta)^2. 
\end{align}
Note that there is only a finite number of possible $r_i$. 
By (\ref{Hodge}), (\ref{below}), 
the possible $\beta_1$ are also finite. 
 Hence the possible $t\in [0, 1)$ is finite. 
\end{proof}

\begin{lem}\label{lem:F}
For $\gamma=(r, \beta, s) \in \Lambda$
with $r\ge 2$ and $F \cdot \beta=1$,
we 
have $\mathfrak{M}_{H_t}^{ss}(\gamma)=\emptyset$
for $t \to 1-0$.  
\end{lem}
\begin{proof}
By Lemma~\ref{lem:wall}, 
the moduli stack 
 $\mathfrak{M}_{H_t}^{ss}(\gamma)$
for $t\to 1-0$ is well-defined. 
Also if there is $[E] \in \mathfrak{M}_{H_t}^{ss}(\gamma)$
for $t \to 1-0$,  
then it must be $F$-slope semistable. 
By~\cite[Lemma~4.3]{Mrule}, any $F$-slope 
semistable sheaf on $\wP$
is restricted to a semistable 
sheaf on a generic fiber of $\wP \to \mathbb{P}^1$. 
Since there is no semistable sheaf on $\mathbb{P}^1$
 with rank bigger than or equal to two
and degree one, we have $\mathfrak{M}_{H_t}^{ss}(\gamma)=\emptyset$
for $t \to 1-0$. 
\end{proof}

By combining Theorem~\ref{thm:Joy4}, Proposition~\ref{prop:blow}
and the above two lemmas, we show the following: 
\begin{prop}\label{prop:reduction}
For any $r \in \mathbb{Z}_{\ge 2}$ and $l\in \mathbb{Z}$, we 
have the following formula: 
\begin{align}\notag
\DT(r, l)=&\sum_{\begin{subarray}{c}
m\ge 2, \ r_1, \cdots, r_m \in \mathbb{Z}_{\ge 1} \\
r_1+ \cdots + r_m=r
\end{subarray}}
\sum_{\begin{subarray}{c}
\overline{\beta}_i =(l_i, a_i) \in \NS_{<r_i}(\wP) \\
1\le i\le m
\end{subarray}}
\sum_{G \in G(m)} \frac{(-1)^{m}}{2^{m-1}} \\
\label{formula:reduction}
&\cdot 
U_{(r_1, \overline{\beta}_1), \cdots, (r_m, \overline{\beta}_m)}^{l, G}(q) 
\cdot
\vartheta_{r, 1-l}(q)^{-1}
\prod_{i=1}^{m} \vartheta_{r_i, a_i}(q) \cdot
\prod_{i=1}^{m} \DT(r_i, l_i). 
\end{align}
\end{prop}
\begin{proof}
We apply Theorem~\ref{thm:Joy4} 
for $S=\wP$, $(H_1, H_2)=(H_0, H_t)$
with $t \in (0, 1)$, 
and $\gamma=(r, \beta, s)$
with $\beta=(l, 1-l)$. 
Using (\ref{pairing}) and (\ref{WCF}), 
we obtain the identity\footnote{It is straightforward to 
generalize the result of Theorem~\ref{thm:Joy4} for 
non-ample $H_0$. }:  
\begin{align}\notag
\Eu_{H_t}(r, \beta, s )=
\sum_{\begin{subarray}{c}
m\ge 1, \ \gamma_i=(r_i, \beta_i, s_i) \in \Lambda \\
\gamma_1+ \cdots + \gamma_m=\gamma  
\end{subarray}}
&\sum_{G \in G(m)}
\frac{1}{2^{m-1}}U(\{\gamma_i\}_{i=1}^{m}, H_0, H_t)
\\ \label{wall:E}
& \cdot \prod_{i \to j \text{ \rm{in} }G} K_{\wP}(r_j \beta_i - r_i \beta_j)
 \prod_{i=1}^{m} \Eu_{H_0}(r_i, \beta_i, s_i). 
\end{align}
Since $F \cdot \beta=1$, we have 
\begin{align*}
\lim_{t\to 1-0}
\Eu_{H_t}(r, \beta, s)=0
\end{align*}
 by Lemma~\ref{lem:F}. 
Therefore by taking $t \to 1-0$
of both sides of (\ref{wall:E}), 
and moving the $m=1$ term to the LHS, 
 we obtain the identity:
\begin{align*}
&\sum_{s} \Eu_{H_0}(r, \beta, s)q^{\frac{\beta^2}{2r}-s}=
-\sum_{\begin{subarray}{c}
m\ge 2, \ r_1, \cdots, r_m 
\in \mathbb{Z}_{\ge 1} \\
r_1 + \cdots + r_m=r
\end{subarray}}
\sum_{G \in G(m)} \frac{1}{2^{m-1}} \\
&\sum_{\begin{subarray}{c}
\beta_i \in \NS(\wP), \ 1\le i\le m \\
\beta_1 + \cdots + \beta_m=\beta
\end{subarray}}
U(\{(r_i, \beta_i)\}_{i=1}^{m}, H, F_{+})
 \prod_{i \to j \text{ in }G}
K_{\wP}(r_j \beta_i - r_i \beta_j)q^{\frac{\beta^2}{2r}-\sum_{i=1}^{m}
\frac{\beta_i^2}{2r_i}} \\
&\hspace{65mm} \prod_{i=1}^{m} \left( \sum_{s_i}
\Eu_{H_0}(r_i, \beta_i, s_i) q^{\frac{\beta^2_i}{2r_i}-s_i} \right). 
\end{align*}
For $\beta_i \in \NS(\wP)$, let 
$\overline{\beta}_i \in \NS_{<r_i}(\wP)$ be the unique 
element such that 
\begin{align*}
\beta_i \equiv \overline{\beta}_i \ (\text{\rm{mod} } r_i).
\end{align*}
Then by (\ref{DTEuH}) and Lemma~\ref{lem:modr}, we have 
\begin{align*}
\sum_{s_i}
\Eu_{H_0}(r_i, \beta_i, s_i) q^{\frac{\beta^2_i}{2r_i}-s_i}
=(-1)^{r_i^2 +1} \DT_{H_0}(r_i, \overline{\beta}_i). 
\end{align*}
Also noting that 
\begin{align}\label{noting}
\frac{\beta^2}{2r}-\sum_{i=1}^{m}
\frac{\beta_i^2}{2r_i}=-\sum_{1\le i<j \le m} \frac{(r_j \beta_i -r_i \beta_j)^2}{2r r_i r_j}
\end{align}
we obtain the following identity: 
\begin{align*}
\DT_{H_0}(r, \beta)=\sum_{\begin{subarray}{c}
m\ge 2, \ r_1, \cdots, r_m \in \mathbb{Z}_{\ge 1} \\
r_1+ \cdots + r_m=r
\end{subarray}}
&\sum_{\begin{subarray}{c}
\overline{\beta}_i \in \NS_{<r_i}(\wP), \\
1\le i\le m.
\end{subarray}}
\sum_{G \in G(m)} \frac{(-1)^{m}}{2^{m-1}} \\
&\cdot 
U_{(r_1, \overline{\beta}_1), \cdots, (r_m, \overline{\beta}_m)}^{l, G}(q) 
\prod_{i=1}^{m} \DT_{H_0}(r_i, \overline{\beta}_i). 
\end{align*}
Applying (\ref{blow:DT}) to both sides of the above 
identity, we obtain the desired formula (\ref{formula:reduction}). 
\end{proof}
We have the following corollary which, 
together with Lemma~\ref{lem:element}, prove
Theorem~\ref{intro:main}. 
\begin{cor}\label{cor:main}
For any $r\in \mathbb{Z}_{\ge 1}$ and $l\in \mathbb{Z}$, there
exist classical data $\xi'$ such that 
\begin{align*}
q^{-\frac{r}{8}}
\eta(q)^{3r} \cdot 
\Theta_{\xi'}(q) \cdot \DT(r, l) \in \mM. 
\end{align*}
\end{cor}
\begin{proof}
The case of $r=1$ follows from (\ref{DT1l}). 
The case of $r\ge 2$ follows from the induction of $r$ 
by Proposition~\ref{prop:SU} and Proposition~\ref{prop:reduction}, noting 
that $\vartheta_{r, a}(q)$ is a classical theta series.  
\end{proof}

\subsection{Proof of Proposition~\ref{prop:SU} for (\ref{S(q)}).}
\label{subsec:propS}
In this subsection, we show that
\begin{align}\label{SinM}
S_{(r_1, \overline{\beta}_1), \cdots, (r_m, \overline{\beta}_m)}^{l, G}(q) 
\in \mM. 
\end{align}
We first prepare the following lemma: 
\begin{lem}\label{lem:F0}
For $l\in \mathbb{Z}$, 
there
are no $r_1, r_2 \in \mathbb{Z}_{\ge 1}$, 
$\beta_1, \beta_2 \in \NS(\wP)$ such that
$\beta_1+\beta_2=(l, 1-l)$ and 
\begin{align}\label{ratio}
\frac{\beta_1}{r_1} \cdot F =\frac{\beta_2}{r_2} \cdot F.
\end{align}
\end{lem}
\begin{proof}
Suppose that there exist such $(r_i, \beta_i)$. 
By substituting $\beta_2=(l, 1-l) -\beta_1$ into 
(\ref{ratio}), and noting $(l, 1-l) \cdot F=1$, 
we obtain 
$\beta_1 \cdot F=r_1/(r_1+r_2)$. Since the LHS is an integer
and the RHS is not an integer, this is a contradiction. 
\end{proof}
We describe the series (\ref{S(q)}) in terms of 
the theta type series in Subsection~\ref{subsec:theta}. 
In the sum (\ref{S(q)}), we set 
\begin{align}\label{ui}
\beta_i=\overline{\beta}_i+r_i u_i
\end{align}
for $u_i \in \NS(\wP)$ and 
\begin{align}\notag
\nu_i &\cneq \frac{\beta_i}{r_i}-\frac{\beta_{i+1}}{r_{i+1}} \\
\label{numu}
&= \overline{\nu}_i
 + u_i -u_{i+1} 
\end{align}
for $1\le i\le m-1$. 
Here we have set
\begin{align*}
\overline{\nu}_i = 
\frac{\overline{\beta}_i}{r_i} - \frac{\overline{\beta}_{i+1}}{r_{i+1}}. 
\end{align*}
Then we have 
\begin{align}\label{nui}
\nu_i \in \overline{\nu}_{i} +\NS(\wP).
\end{align}
Conversely given $\nu_i$ for $1\le i\le m-1$ as in 
(\ref{nui}), 
we can write $u_i$ satisfying (\ref{numu})
as follows: 
\begin{align*}
u_i=u_1 -(\nu_1 -\overline{\nu}_1) - \cdots -
 (\nu_{i-1} -\overline{\nu}_{i-1}). 
\end{align*}
By substituting into (\ref{ui}) and 
$\beta_1 + \cdots + \beta=\beta$, 
where $\beta=(l, 1-l)$, we have
\begin{align*}
\beta-\overline{\beta}
+ \sum_{i=1}^{m-1} \sum_{j=i+1}^{m} r_j (\nu_{i}-\overline{\nu}_i) = ru_1. 
\end{align*}
Here we have set 
$\overline{\beta}=\overline{\beta}_1+ \cdots + \overline{\beta}_m$. 
Therefore the necessary and sufficient condition for $\nu_i$ in (\ref{nui})
to have the solution $(u_1, \cdots, u_m) \in \NS(\wP)^{\times m}$ is 
\begin{align*}
\beta-\overline{\beta}
+ \sum_{i=1}^{m-1} \sum_{j=i+1}^{m} r_j (\nu_{i}-\overline{\nu}_i)
\in r \NS(\wP). 
\end{align*}
On the other hand, for each $1\le i\le m-1$, we have
\begin{align}\label{sgn}
&\sgn \left( F \cdot \left( \frac{\beta_1+ \cdots + \beta_i}{r_1+ \cdots + r_i} -\frac{\beta_{i+1}+ \cdots + \beta_m}{r_{i+1}+ \cdots + r_m} \right)
\right) \\
\notag
&=\sgn \left( F \cdot \sum_{k\le i<j} (\beta_k r_j -\beta_j r_k)
 \right)  \\
\notag
&=\sgn \left( F\cdot 
\sum_{k\le i<j}r_j r_k(\nu_k + \nu_{k+1} + \cdots + \nu_{j-1})
 \right). 
\end{align}
By Lemma~\ref{lem:F0}, 
the value (\ref{sgn}) is non-zero 
for $(r_i, \beta_i)$ in the series (\ref{S(q)}). 
Therefore
the series (\ref{S(q)}) is written as 
\begin{align}\label{S1q}
&\prod_{i \to j \text{ in }G} r_i r_j 
\cdot \frac{1}{2^{m-1}}\sum_{\begin{subarray}{c}
\nu_i \in \overline{\nu}_i + \NS(\wP), \ 1\le i\le m-1 \\
\beta-\overline{\beta}
+ \sum_{i=1}^{m-1} \sum_{j=i+1}^{m} r_j (\nu_{i}-\overline{\nu}_i)
\in r \NS(\wP)
\end{subarray}} \\
\notag
&\prod_{i=1}^{m-1}
\left( \sgn(H_0 \cdot \nu_i) - \delta_{0, H_0 \cdot \nu_i} -
\sgn \left(F \cdot \sum_{k\le i<j}r_j r_k(\nu_k + \nu_{k+1} + \cdots + \nu_{j-1})
 \right)  \right) \\
\notag
& \cdot \prod_{i \to j \text{ in } G} K_{\wP}(\nu_i+\nu_{i+1}+
 \cdots + \nu_{j-1}) \cdot 
q^{-\sum_{1\le i<j\le m}\frac{r_i r_j(\nu_i+\nu_{i+1}+
 \cdots + \nu_{j-1})^2}{2r}}. 
\end{align}
We set
\begin{align*}
\Gamma=\left\{(\nu_1, \cdots, \nu_{m-1}) \in \NS(\wP)^{\times m-1} : 
\sum_{i=1}^{m-1} \sum_{j=i+1}^{m} r_j \nu_{i}
\in r \NS(\wP) \right\}. 
\end{align*}
Let $\overline{\nu}'$ be one of 
$(\nu_1, \cdots, \nu_{m-1})$ in the series (\ref{S1q}). 
Since $\Gamma \subset \NS(\wP)^{\times m-1}$ is of finite index, 
we have $\overline{\nu}' \in \Gamma_{\mathbb{Q}}$. 
By expanding,  
the series (\ref{S1q}) is a linear combination of the series
\begin{align}\notag
&\sum_{\begin{subarray}{c}
\nu \in \overline{\nu}' + \Gamma \\
H_0 \cdot \nu_i=0 \\
\text{ for all } i\in T
\end{subarray}} \prod_{i\notin T}
\left( \sgn(H_0 \cdot \nu_i) -
\sgn \left(F \cdot \sum_{k\le i<j}r_j r_k(\nu_k + \nu_{k+1} + 
\cdots + \nu_{j-1})
 \right)  \right) \\
\label{S2q}
& \hspace{13mm}
\prod_{i \to j \text{ in } G} K_{\wP}(\nu_i+\nu_{i+1}+
 \cdots + \nu_{j-1}) \cdot 
q^{-\sum_{1\le i<j\le m}\frac{r_i r_j(\nu_i+\nu_{i+1}+
 \cdots + \nu_{j-1})^2}{2r}}. 
\end{align}
Here $T \subset \{1, \cdots, m-1\}$ is a subset. 
By Lemma~\ref{lem:classical}, it is enough
to show that the series (\ref{S2q}) with $T=\emptyset$ 
is written as 
$\Theta_{\xi}(q^{\frac{1}{r}})$ for data $\xi$ as in (\ref{xi}). 

Let $A=\{a_{ij}\}_{1\le i, j \le m-1}$ be the $(m-1)\times (m-1)$-matrix
given by
\begin{align*}
a_{ij}=\left\{ \begin{array}{cl} -\sum_{k\le i, j<l}r_l r_k & (i\le j) \\
a_{ji} & (i>j).
\end{array}
\right.
\end{align*}
We define the integer valued
symmetric bilinear pairing $B(-, -)$ on $\Gamma$
by 
\begin{align*}
B(\nu, \nu')=\nu \cdot A \cdot {}^t \nu'. 
\end{align*}
Here we regard an element $\nu \in \Gamma$
as a row vector $(\nu_1, \cdots, \nu_{m-1})$
in $\NS(\wP)^{\times m-1}$. 
It is straightforward to check that 
\begin{align}\notag
Q(\nu) &\cneq \frac{B(\nu, \nu)}{2} \\
\label{Qnu}
&=
-\sum_{1\le i<j\le m}\frac{r_i r_j(\nu_i+\nu_{i+1}+ \cdots + \nu_{j-1})^2}{2}.
\end{align}
Since $\NS(\wP)$ with its intersection 
form is a lattice with index $(1, 1)$, 
it follows that $(\Gamma, B(-, -))$ is a non-degenerate lattice
with index $(m-1, m-1)$. In particular, we have $\det A \neq 0$. 

We set $c_i$, $c_i' \in \Gamma_{\mathbb{Q}}$ for $1\le i\le m-1$
as follows:
\begin{align}\label{ci}
&c_i=(0, \cdots, 0, \stackrel{i}{H_0}, 0, \cdots, 0) A^{-1} \\
\notag
&c_i'=(0, \cdots, 0, \stackrel{i}{-F}, 0, \cdots, 0). 
\end{align}
Here $\stackrel{i}{\ast}$ means that 
$\ast$ is located on the $i$-th column. 
Let $E(G)$ be the set of arrows in $G$, and take $e=(i \to j) \in E(G)$. 
Since $(\Gamma, B)$ is non-degenerate, 
there exists $\alpha_e \in \Gamma_{\mathbb{Q}}$ such that
\begin{align*}
B(\alpha_e, \nu)=K_{\wP}(\nu_i+\nu_{i+1}+ \cdots + \nu_{j-1})
\end{align*}
for any $\nu \in \Gamma_{\mathbb{Q}}$. 
By the above constructions, 
the series (\ref{S2q})
with $T=\emptyset$ is written 
in the following way:
\begin{align*}
\sum_{\nu \in \overline{\nu}' + \Gamma}
\prod_{i=1}^{m-1}
\left( \sgn (B(c_i, \nu))-
\sgn ( B(c_i', \nu) ) \right)
\prod_{e \in E(G)}
B(\alpha_e, \nu)
 \cdot q^{\frac{Q(\nu)}{r}}. 
\end{align*} 
Hence the following lemma shows that (\ref{SinM}) holds:
\begin{lem}
The data
\begin{align*}
(\Gamma, B, \overline{\nu}', 
c_1, \cdots, c_{m-1}, c_{1}', \cdots, c_{m-1}', \alpha_{e}, 
e\in E(G))
\end{align*}
satisfies the conditions (i) to (v) in Subsection~\ref{subsec:theta}. 
\end{lem}
\begin{proof}
The condition (i) is already stated. 
Let $V \subset \Gamma_{\mathbb{Q}}$ be the
sub $\mathbb{Q}$-vector space spanned by 
$c_i$ for $1\le i\le m-1$. 
By (\ref{ci}), $V$ is $m-1$-dimensional, and 
\begin{align*}
V=\bigoplus_{i=1}^{m-1} 
\mathbb{Q}\cdot (0, \cdots, 0, \stackrel{i}{H_0}, 0, \cdots, 0).
\end{align*}
By (\ref{Qnu}), it follows that 
$Q$ is negative definite on $V$, hence
the condition (ii) holds. The condition (iii) follows from 
\begin{align*}
B(c_i, c_j')&=(0, \cdots, \stackrel{i}{H_0}, \cdots, 0) \cdot 
{}^t (0, \cdots, \stackrel{j}{-F}, \cdots, 0) \\
&=-\delta_{ij} \\
B(c_i', c_j')&=
(0, \cdots, \stackrel{i}{-F}, 
\cdots, 0) \cdot {}^t (b_1 F, \cdots, b_{m-1}F) \\
&=0. 
\end{align*}
Here $b_1, \cdots, b_{m-1}$ are some rational numbers, 
and the last equality follows from $F^2=0$. 
The condition (iv) follows from Lemma~\ref{lem:F0}, 
and there is nothing to prove for (v).  
\end{proof}

\subsection{Proof of Proposition~\ref{prop:SU} for (\ref{W(q)})}
\label{subsec:propU}
We finish the proof of Proposition~\ref{prop:SU} by proving that 
\begin{align}\label{WinM}
U_{(r_1, \overline{\beta}_1), \cdots, (r_m, \overline{\beta}_m)}^{l, G}(q) 
\in \mM. 
\end{align}
By Lemma~\ref{lem:F0}
and (\ref{m''}), 
the rational number 
$U(\{(r_i, \beta_i)\}_{i=1}^{m}, H, F_{+})$
in the RHS of (\ref{W(q)}) does not 
contain contributions of the terms in (\ref{defi:U}) with $m'' \ge 2$. 
For data (\ref{data:W}) and a fixed non-decreasing surjection
\begin{align*}
\psi \colon \{1, \cdots, m \} \twoheadrightarrow \{1, \cdots, m'\}
\end{align*}
we consider the series:
\begin{align}\notag
S_{(r_1, \overline{\beta}_1), \cdots, 
(r_m, \overline{\beta}_m)}^{l, G, \psi}
\cneq \sum_{\begin{subarray}{c}
\beta_j \in \NS(\wP), \
\beta_j \equiv \overline{\beta}_j \ (\text{\rm{mod} } r_j) \\
1\le j\le m, \
\beta_1 + \cdots + \beta_m=(l, 1-l) \\
\beta_j \cdot H_0/r= \beta_k \cdot H_0/r \text{ if }
\psi(j)=\psi(k)
\end{subarray}}
S(\{(R_i, \bB_i) \}_{i=1}^{m'}, H, F_{+} ) \\
\label{W(q)2}
\cdot \prod_{i \to j \text{ in } G} K_{\wP}(r_j \beta_i - r_i \beta_j) q^{-\sum_{1\le i<j \le m} \frac{(r_j \beta_i -r_i \beta_j)^2}{2r r_i r_j}}.
\end{align}
Here $(R_i, \bB_i)$ is given by 
\begin{align}\label{dag}
R_i=
\sum_{j \in \psi^{-1}(i)}r_j, \quad
\bB_i=\sum_{j \in \psi^{-1}(i)}\beta_j.
\end{align}
Since (\ref{W(q)}) is a $\mathbb{Q}$-linear combination of the series 
(\ref{W(q)2}), 
it is enough to show that (\ref{W(q)2}) 
is an element of $\mM$.  

Let us consider $(\beta_1, \cdots, \beta_m)$ in the RHS of (\ref{W(q)2}). 
For each $1\le j \le m$
with $\psi(j)=i$, 
we can write
\begin{align}\label{cond:write}
\frac{\beta_j}{r_j}=\frac{\bB_i}{R_i} + l_j C
\end{align}
for some $l_j \in \mathbb{Q}$. 
For $1\le i\le m'$, let 
$\overline{\bB}_{i} \in \NS_{<R_i}(\wP)$
be the unique element such that 
\begin{align*}
\bB_i \equiv \overline{\bB}_i \ (\text{\rm{mod} } R_i).
\end{align*}
Then we have
\begin{align}\label{cond:mj}
l_j C \in \frac{\overline{\beta}_j}{r_j} -
\frac{\overline{\bB}_i}{R_i} +\NS(\wP). 
\end{align}
By applying $\cdot H_0$ and $\cdot C$, the condition (\ref{cond:mj})
is equivalent to the two conditions:
\begin{align*}
\frac{\overline{\bB}_i}{R_i} \cdot H_0
\in \frac{\overline{\beta}_j}{r_j} \cdot H_0 + \mathbb{Z}, \quad 
l_j \in \frac{\overline{\bB}_i}{R_i} \cdot C
-\frac{\overline{\beta}_j}{r_j} \cdot C + \mathbb{Z}. 
\end{align*}
Also using (\ref{cond:write}), the condition (\ref{dag}) for $\bB_i$ 
is equivalent to 
\begin{align*}
\sum_{j \in \psi^{-1}(i)} r_j l_j=0. 
\end{align*}
Using (\ref{cond:write})
and noting $K_{\wP} \cdot C=-1$, we have
\begin{align*}
\prod_{i \to j \text{ in }G}
K_{\wP}(r_j \beta_i - r_i \beta_j) = 
&\prod_{i\to j \text{ in }G} r_i r_j \cdot 
\sum_{\begin{subarray}{c}G', G'' \subset G \\
E(G') \coprod E(G'')=E(G)
\end{subarray}} \\
&\prod_{i \to j \text{ in }G'}(l_j -l_i) 
\cdot \prod_{ i\to j \text{ in } \psi(G'')}
K_{\wP} \left( \frac{\bB_i}{R_i}
-\frac{\bB_j}{R_j}\right). 
\end{align*}
Here $G', G'' \subset G$ are oriented 
subgraphs, $E(G)$ is the set of arrows in $G$, 
and $\psi(G'') \in G'(m')$ is obtained 
by the image of $G''$ under $\psi$. 
Also using (\ref{noting}) and setting $\beta=(l, 1-l)$, 
the power of $q$ in (\ref{W(q)2}) 
is written as 
\begin{align*}
&\frac{\beta^2}{2r}-\sum_{i=1}^{m}
\frac{\beta_i^2}{2r_i} \\
&=\frac{\beta^2}{2r}-\sum_{i=1}^{m'}
\frac{\bB_i^2}{2R_i} + 
\sum_{i=1}^{m'}\left( \frac{\bB_i^2}{2R_i} - \sum_{j \in \psi^{-1}(i)}
\frac{\beta_j^2}{2r_j} \right) \\
&=-\sum_{1\le i<j\le m'} \frac{(R_j\bB_i -
R_i \bB_j)^2 }{2r R_i R_j}
+\sum_{i=1}^{m'} \sum_{j<k \text{ in } \psi^{-1}(i)}
\frac{r_j r_k(l_k-l_j)^2}{2R_i}. 
\end{align*}
Combing the above calculations, the 
series (\ref{W(q)2}) is written as
\begin{align*}
&\prod_{i \to j \text{ in }G} r_i r_j \cdot 
\sum_{\begin{subarray}{c}G', G'' \subset G \\
E(G') \coprod E(G'')=E(G)
\end{subarray}}
\sum_{\begin{subarray}{c}
\overline{\bB}_i \in \NS_{< R_i}(\wP), \ 1\le i\le m' \\
\overline{\bB}_i \cdot H_0/R_i
\in \overline{\beta}_j \cdot H_0/r_j+ \mathbb{Z}
\text{ for all } j \in \psi^{-1}(i) 
\end{subarray}} \\
& \cdot 
\sum_{\begin{subarray}{c}
l_j \in \overline{\bB}_{\psi(j)} \cdot C/R_{\psi(j)}
-\overline{\beta}_j \cdot C/r_j + \mathbb{Z}, \
1\le j\le m \\
\sum_{j \in \psi^{-1}(i)}r _j l_j=0 \text{ for all }
1\le i\le m'
\end{subarray}}
\prod_{i \to j \text{ in } G'}
(l_j-l_i)  \\
& \hspace{60mm}
 q^{\sum_{i=1}^{m'} \sum_{j<k \text{ in } \psi^{-1}(i)}
\frac{r_j r_k(l_k-l_j)^2}{2R_i}}  \\
&\cdot \sum_{\begin{subarray}{c}
\bB_i \in \NS(\wP), \ \bB_i \equiv \overline{\bB}_i \
(\text{mod } R_i) \\
\bB_1+ \cdots + \bB_{m'}=(l, 1-l) 
\end{subarray}}
S(\{(R_i, \bB_i)\}_{i=1}^{m'}, H, F_{+})
\prod_{ i\to j \text{ in } \psi(G'')}
K_{\wP} \left( \frac{\bB_i}{R_i}
-\frac{\bB_j}{R_j}\right) \\
& \hspace{70mm} 
q^{-\sum_{1\le i<j\le m'} \frac{(R_j\bB_i -
R_i \bB_j)^2 }{2r R_i R_j}}. 
\end{align*}
By Lemma~\ref{lem:classical}, the series 
\begin{align*}
&\sum_{\begin{subarray}{c}
l_j \in \overline{\bB}_{\psi(j)} \cdot C/R_{\psi(j)}
-\overline{\beta}_j \cdot C/r_j + \mathbb{Z}, \
1\le j\le m \\
\sum_{j \in \psi^{-1}(i)}r _j l_j=0 \text{ for all }
1\le i\le m'
\end{subarray}}
\prod_{i \to j \text{ in } G'}
(l_j-l_i)  \\
& \hspace{60mm}
 q^{\sum_{i=1}^{m'} \sum_{j<k \text{ in } \psi^{-1}(i)}
\frac{r_j r_k(l_k-l_j)^2}{2R_i}}
\end{align*}
is an element of $\mM$. Combined with (\ref{SinM}), 
we conclude that (\ref{W(q)2}) is an element of $\mM$. 

\providecommand{\bysame}{\leavevmode\hbox to3em{\hrulefill}\thinspace}
\providecommand{\MR}{\relax\ifhmode\unskip\space\fi MR }
\providecommand{\MRhref}[2]{%
  \href{http://www.ams.org/mathscinet-getitem?mr=#1}{#2}
}
\providecommand{\href}[2]{#2}

Kavli Institute for the Physics and 
Mathematics of the Universe, University of Tokyo,
5-1-5 Kashiwanoha, Kashiwa, 277-8583, Japan.

\textit{E-mail address}: yukinobu.toda@ipmu.jp

\end{document}